\documentclass[11pt]{amsart}
\usepackage{graphicx}
\usepackage{amssymb}
\usepackage{epstopdf}
\usepackage{MnSymbol}

\newtheorem{lemma}{Lemma}[section] 

\newtheorem{prop}[lemma]{Proposition}
\newtheorem{example}[lemma]{Example}
\newtheorem{theorem}[lemma]{Theorem}
\newtheorem{cor}[lemma]{Corollary}

\newtheorem{defn}[lemma]{Definition}
\newtheorem{remark}[lemma]{Remark}

\newcommand{\C}{\mathbb{C}}
\newcommand{\R}{\mathbb{R}}

\newcommand{\extd}{\mathrm{d}}
\newcommand{\di}{\,\mathrm{d}}

\newcommand{\tens}{\mathop{\otimes}}

\newcommand{\<}{\langle}
\renewcommand{\>}{\rangle}

\title[From homotopy to It\^o calculus and Hodge theory]{From homotopy to It\^o calculus and Hodge theory 
}
\author{G.\ Alhamzi, E.J.\,Beggs \& A.D.\ Neate}

\date{}                                           

\begin{document}
\maketitle

\begin{abstract}

We begin with a deformation of a differential graded algebra by adding time and using a homotopy. It is shown that the standard formulae of It\^o calculus are an example, with four caveats: First, it says nothing about probability. Second, it assumes smooth functions. Third, it deforms all orders of forms, not just first order. Fourth, it also deforms the product of the DGA. 
An isomorphism between the deformed and original DGAs may be interpreted as the  transformation rule between the Stratonovich and classical calculus (again no probability). The isomorphism can be used to construct covariant derivatives with the deformed calculus.
We apply the deformation in noncommutative geometry, to the Podle\'s sphere $S^2_q$. This involves the Hodge theory of $S^2_q$. 

\end{abstract}

\section{Introduction}

In this paper we consider a deformation of a differential calculus. The idea is very general, but we try to keep to concrete examples that should be familiar to anyone with a knowledge of differential calculus on manifolds. The exception is that we later look at the calculus on the noncommutative sphere $S^2_q$. This is done partly out of interest, and partly to make a point: There is no assumption that the `space' behind the construction should be an ordinary manifold, we can carry out the construction in \textit{noncommutative geometry}. As this paper crosses more than one subject, we will endeavour to explain ideas clearly, and we apologise to experts in advance for this. 

The idea of a \textit{differential graded algebra} (DGA for short) may be best expressed to most readers by taking the usual de Rham complex of differential forms on a manifold, with differential $\extd$ and product $\wedge$. However it is familiar in different subjects with other examples, e.g.\ the singular cochains of algebraic topology or the Dolbeault complex of algebraic geometry. The method we describe applies to such cases, but it is not obvious to the authors what examples would be interesting for these cases, so we concentrate on the de Rham complex. 

The ingredients for the deformation are a DGA with a \textit{homotopy} $\delta$ (i.e.\ a map on the DGA reducing degree by one). Such a homotopy on a DGA is directly related to the idea of homotopy in topology, which is how cohomology theories are shown to be homotopy invariant. First an extra coordinate, which we label `time', is added in an essentially trivial manner, and then the extended DGA is deformed by the homotopy. As an example, we construct the well known  It\^o calculus for diffusions as a homotopy deformation of the usual calculus. 
 It should be noted that by `constructing' the It\^o calculus we simply mean obtaining the formulae of It\^{o} calculus, there is no idea of a probabilistic derivation. 
However the homotopy deformation gives a calculus to all orders of differential forms, and deforms both the differential and the product. We get a differential which has properties in common with the It\^o differential, but has no probabilistic interpretation. Together with the deformed product we get a DGA isomorphic to the classical calculus, and we argue that this corresponds to the Stratonovich integral. Note that this deformed product is commutative.
This isomorphism is used in Section~\ref{vcaisuvfiuyt} to construct a covariant derivative for the deformed calculus. 
In Section \ref{bcuyowiiygccjbehb} we consider a novel application of this approach to proving the path independence of the Girsanov change of measure \cite{Truman} , which from our point of view is interpreted as a cohomology calculation, using the invariance of cohomology to the deformation given in Corollary~\ref{vcyudksyufvj}.

From Section~\ref{vcadisuvuid} we begin the second part of the paper, which is a noncommutative application of the first part and its relation to Hodge theory.
As a specific example of the deformation, we use the standard Podle\'s noncommutative sphere $S^2_q$ \cite{pod87}. The initial reason is that the machinery allows us to take noncommutative examples, and the classical cases have been covered in many places. Indeed we do get a Laplace operator $\Delta$ on $S^2_q$, and we can list 
the eigenfunctions for $\Delta$. This is not new, it was done in  \cite{pod87}. However the homotopy machinery should allow us to calculate $\Delta$ for all forms, and write the corresponding eigen-$n$-forms. However when we try, we find highly non-unique results. The problem is that the formula for $\delta$ involves an interior product $X\,\righthalfcup\,\xi\in \Omega^{n-1}S^2_q$ for a vector field $X$ and $\xi\in\Omega^nS^2_q$. There is no problem defining vector fields, but there is, as yet, no well defined idea (i.e.\ no sensible unique definition) of the interior product for $n>1$. 

There is a way out of this problem, involving Hodge theory. We follow the classic account of Hodge theory in \cite{GHalgGeom, vois}. We give an explicit Hodge operation $\diamondsuit$ (we do not use the star symbol for this to avoid confusion with the star of $C^*$ algebras, which extends to forms in a different way). Now we can insist that $\delta$ is the \textit{codifferential} corresponding to $\extd$ under the Hodge operation.  Doing this allows us to find a formula for the interior product for $n>1$, and so we can list the eigen-$n$-forms. But now we can essentially do all of (real) Hodge theory on $S^2_q$, including the Hodge operation, the pairings of forms, and the projections to harmonic forms. This should not be too surprising, as general Hodge theory involves elliptic operators, and there are associated heat kernel methods. The idea, which may have some chance of extension to noncommutative geometry in some level of generality, is the following: If the eigenvalues of $\Delta$ are all of one sign, then the heat diffusion of the $n$-forms should `tend to' (insert appropriate convergence) a harmonic form (i.e.\ in the zero eigenspace). However, if we start with a closed form, the de Rham cohomology class of the $n$-form will be conserved, so we should get a projection to harmonic forms preserving the cohomology. The problem is carrying out the analysis in any generality for the noncommutative theory. 

We have so far not mentioned probability. For an introduction to stochastic analysis and It\^{o} calculus for diffusions there are many standard texts for instance \cite{mckeanStoInt, IkeWat}. This theory has well known extensions to differential  manifolds, see \cite{ElJanLi,Elworthy1362,ElworthyCUP,IkeWat}.    For a geometric discussion on It\^o calculus, see \cite{MeyerSto}. 
There are existing ideas of It\^o calculus and Brownian motion in noncommutative geometry, and it would be interesting to see how they fit with the idea of homotopy deformation. There is a non-commutative theory of \emph{quantum stochastic differential equations} developed by Hudson and Parthasarathy \cite{partha} which has connections to Hopf algebras, see \cite{Hudson05,Hudson12,HudPar}. For quantum Brownian motions on quantum homogeneous spaces, see \cite{dasGos}.  Quantum stochastic
processes on the  noncommutative torus and  Weyl $C^*$ algebra are considered in  \cite{ChaGosSin}. It is useful to compare this paper to \cite{DimMulSto}, where a Moyal type product is used, resulting in a noncommutative calculus, which is then related to the It\^o calculus, including higher order differential forms. 
We emphasise that in this paper we are not performing It\^o calculus by noncommutative geometry. If the original DGA is graded commutative, as is the case for the classical de Rham complex, then the homotopy deformed DGA remains graded commutative. (The example of the noncommutative 2-sphere is not graded commutative of course, but then it was not graded commutative before the homotopy deformation.) In \cite{majid12} there is also a deformation related to It\^o calculus, but the motivation there is to deform the de Rham complex to a noncommutative DGA. Noncommutative heat kernels are considered in \cite{VassHeat,Gordina,ChaSun,CaGaReSu}.

Lest it be thought that the problem of the definition of the interior product was of little consequence, a special case of it is essentially the same problem as defining the Ricci curvature in terms of the Riemann curvature, which is directly calculable from a covariant derivative. As there is, as yet, no general method to calculate the Ricci curvature in noncommutative geometry, there is no direct way to write the Einstein equations of general relativity in noncommutative geometry. As one possibility for combining quantum theory and gravity is to use noncommutative geometry, this is a problem.

There are several studies of Hodge theory in noncommutative geometry, see \cite{GaetQGP}.  
A complex linear version of the Hodge operator for the noncommutative sphere was given in \cite{majspinsphere}. 
An antilinear Hodge operator for $\mathbb{CP}^2_q$ was given in \cite{DAnLandGeoQuyaProj}. 
For a discussion of Hodge structures in terms of motives, see \cite{CoMaMotiv}. For the relation with mirror symmetry, see \cite{KatKonPan}. 

We should point out a complication with the sign of the Laplacian. The Laplacian usually used in It\^o calculus is the Laplace-Beltrami operator, which is the usual sum of double derivatives on standard $\mathbb{R}^n$. In terms of functional analysis, this is a negative operator. The homotopy construction naturally gives the operator $\Delta=\delta\,\extd+\extd\,\delta$. If $\delta$ is the Hodge theory adjoint of $\extd$, this is a positive operator, the Hodge Laplacian. (On forms rather than functions, there are additional diffferences due to curvature.) In Section~\ref{cbhajksvckvkvv1} for applications to It\^o calculus we specify a $\delta$ giving the Laplace-Beltrami operator, which is therefore not the Hodge theory adjoint of $\extd$. In the noncommutative sphere example we adapt the same formula for $\delta$ which gives the 
Laplace-Beltrami operator. When we look at Hodge theory in Section~\ref{bcdilskjhcj}, we have to reconcile the sign conventions. In fact, as we never have to specify the constants $\alpha,\beta$ (not even their sign) in the metric (\ref{vchdksukjf}) on the noncommutative sphere, we keep the homotopy (\ref{cvdgsyjhtc}) in its original form.


The authors would like to thank Robin Hudson, Xue-Mei Li, Jiang-Lun Wu and Aubrey Truman for their assistance, and the organisers of the LMS Meeting and Workshop `Quantum Probabilistic Symmetries' in Aberystwyth in September 2012, where this work was first presented (though with the torus rather than sphere as a noncommutative example).

\section{Homotopy deformation}

\subsection{Preliminaries on differential graded algebras}
A \textit{differential graded algebra} (DGA for short) $(F^*,\extd,\wedge)$ is given by vector spaces (over $\mathbb{R}$ or $\mathbb{C}$) $F^n$ for $n\ge 0$ (conventionally write $F^n=0$ for $n<0$), a linear map $\extd:F^n\to F^{n+1}$ (the \textit{differential}) and a bilinear map (an associative \textit{wedge product})
$\wedge:F^n\times F^m\to F^{n+m}$.  For $\xi\in F^n$ it will be convenient to write its \textit{grade} as $|\xi|=n$. 
The operations obey the rules
\begin{eqnarray}
\extd\,\extd\,=\,0\ ,\quad \extd(\xi\wedge\eta)\,=\,\extd(\xi)\wedge\eta+(-1)^{|\xi|}\,\xi\wedge\extd(\eta)\ .
\end{eqnarray}
The second rule is the \textit{graded derivation} property for $\extd$.
Note that though we have assumed associativity (i.e.\ $(\xi\wedge\eta)\wedge\zeta=\xi\wedge(\eta\wedge\zeta)$) we do not assume graded commutativity, even though it is true for the de Rham forms on a classical manifold, i.e.\ we do not assume that
\begin{eqnarray}  \label{cvaisuyvguifvui}
\xi\wedge\eta\,=\,(-1)^{|\xi|\,|\eta|}\,\eta\wedge\xi\ .
\end{eqnarray}

The main example we shall consider is $F^n=\Omega^nM$, the $n$-forms on a differential manifold $M$, with the usual differential and wedge product. Note that the product $\wedge:F^0\times F^0\to F^{0+0}$ makes $F^0$ into an algebra, but because we do not assume (\ref{cvaisuyvguifvui}), in general $F^0$ need not be a commutative algebra. For the de Rham complex on a differential manifold $M$, $F^0=\Omega^0M$ is the real or complex valued smooth functions on $M$. 
Note that the general definition of DGA does not need the vector space assumption or $F^n=0$ for $n<0$, but we find both convenient for this paper. 

As $F^0$ is an algebra and we have products $\wedge:F^n\times F^0\to F^{n}$ and $\wedge:F^0\times F^n\to F^{n}$, we have each $F^n$ being a \textit{bimodule} over $F^0$. This is just saying that we can multiply $n$-forms by functions in the usual de Rham complex, but in general, as (\ref{cvaisuyvguifvui}) may not hold, the products on the right and left may be different. Often we write $f.\xi$ or $\xi.f$ instead of using $\wedge$ for $f\in F^0$.

\subsection{Extending the DGA by time}
Given a differential manifold $M$, we can add an extra coordinate to get $M\times\mathbb{R}$. The extra coordinate we call \textit{time} $t$. If we write $F^n=\Omega^n M$, then the de Rham complex of $M\times\mathbb{R}$ has $n$-forms
\begin{eqnarray}\label{bvhvbv}
F_{\mathbb{R}}^n \ =\ F^n\tens C^\infty(\mathbb{R}) \ \bigoplus\ \Big(F^{n-1}
\tens C^\infty(\mathbb{R})\Big)\wedge \extd t\ .
\end{eqnarray}
To explain (\ref{bvhvbv}), the part of the right hand side before the \textit{direct sum} $\oplus$ is the $n$-forms in $\Omega^n(M\times\mathbb{R})$ which have no $\extd t$ component, and those with a $\extd t$ are after the $\oplus$. 
The symbol \textit{tensor product} $\otimes$ can simply be interpreted as `sums of products of'. Thus $F^n\tens C^\infty(\mathbb{R})$ consists of sums of elements of $\Omega^nM$ times functions of time $t\in\mathbb{R}$, so we have simply added variation with respect to time into the forms on $M$. An alternative phrasing would be to say that elements of $F^n\tens C^\infty(\mathbb{R})$ are time dependent elements of $\Omega^nM$, or functions of time taking values in    $\Omega^nM$   (ignoring technicalities on finiteness of sums or completing tensor products).

We take (\ref{bvhvbv}) in the case of a general DGA. Then the $F_{\mathbb{R}}$ have also the structure of a DGA
 $(F_{\mathbb{R}}^*,\extd_0,\wedge_0)$, with operations, for $\xi\in F^n\tens C^\infty(\mathbb{R})$ and $\eta\in F^m\tens C^\infty(\mathbb{R})$,
\begin{eqnarray}\label{bvhvbvyy}
\extd_0(\xi) &=& \extd \xi +(-1)^n\,\frac{\partial \xi}{\partial t}\wedge\extd t \ ,\cr
\extd_0(\xi\wedge\extd t) &=& \extd\xi\wedge\extd t\ ,\cr
\xi\wedge_0\eta &=& \xi\wedge \eta \cr
(\xi\wedge\extd t)\wedge_0 \eta &=& (-1)^m\, (\xi\wedge \eta)\wedge \extd t\ ,\cr
\xi\wedge_0 (\eta\wedge\extd t) &=&  (\xi\wedge \eta)\wedge \extd t\ ,\cr
(\xi\wedge\extd t)\wedge_0(\eta\wedge\extd t) &=& 0\ .
\end{eqnarray}
This is just the usual tensor product of differential graded algebras, where $C^\infty(\mathbb{R})$ has the usual differential calculus.  We should point out that the $\extd$ and $\wedge$ operations
in (\ref{bvhvbvyy}) are from the original $(F^*,\extd,\wedge)$, and are applied at fixed values of $t$ to the forms, so for example $(\extd\xi)(t)=\extd(\xi(t))$ and $(\xi\wedge \eta)(t)=\xi(t)\wedge \eta(t)$. The symbols $\extd_0$ and $\wedge_0$
 are used as we will later have a deformation parameter $\alpha$, and these operations correspond to $\alpha=0$.

\subsection{The homotopy deformation}
A \textit{homotopy} is a linear map $\delta:F^n\to F^{n-1}$ for all $n$ (remembering that $F^n=0$ for $n<0$). The reader may consult textbooks on algebraic topology to see how this corresponds to the idea of homotopy in topology. 
Let $\Delta=\delta\,\extd+\extd\,\delta:F^n\to F^{n}$. 
 By construction $\Delta$ is a cochain map, i.e.\ $\extd\,\Delta=\Delta\,\extd$. There is a deformation
 (with parameter $\alpha\in\mathbb{R}$) $(F_{\mathbb{R}}^*,\extd_\alpha,\wedge_\alpha)$
 of $(F_{\mathbb{R}}^*,\extd_0,\wedge_0)$ given by,
   for $\xi\in F^n\tens C^\infty(\mathbb{R})$ and $\eta\in F^m\tens C^\infty(\mathbb{R})$,
 \begin{eqnarray}\label{bvhvbvzz}
\extd_\alpha(\xi) &=& \extd \xi +(-1)^n\,\frac{\partial \xi}{\partial t}\wedge\extd t + (-1)^n\,\alpha\,\Delta(\xi)\wedge\extd t\ ,\cr
\extd_\alpha(\xi\wedge\extd t) &=& \extd\xi\wedge\extd t\ ,\cr
\xi\wedge_\alpha \eta &=& \xi\wedge \eta - (-1)^{n+m}\,\alpha\,\big(\delta(\xi\wedge \eta)-
\delta(\xi)\wedge \eta \cr
&& -\,(-1)^n\,\xi\wedge \delta(\eta)\big)\wedge\extd t\ , \cr
(\xi\wedge\extd t)\wedge_\alpha \eta &=& (-1)^m\, (\xi\wedge \eta)\wedge \extd t\ ,\cr
\xi\wedge_\alpha (\eta\wedge\extd t) &=&  (\xi\wedge \eta)\wedge \extd t\ ,\cr
(\xi\wedge\extd t)\wedge_\alpha(\eta\wedge\extd t) &=& 0\ .
\end{eqnarray}
 Note that if the original DGA $(F^*,\extd,\wedge)$ is graded commutative (i.e.\ (\ref{cvaisuyvguifvui}) holds), then so is the homotopy deformation
 $(F_{\mathbb{R}}^*,\extd_\alpha,\wedge_\alpha)$.

\begin{theorem}
The operations $\extd_\alpha$ and $\wedge_\alpha$ in (\ref{bvhvbvzz}) make $F_{\mathbb{R}}^*$ into a DGA.
\end{theorem}
\noindent {\bf Proof:}\quad We take $\xi\in F^n\tens C^\infty(\mathbb{R})$ and $\eta\in F^m\tens C^\infty(\mathbb{R})$.
First check that $\extd_\alpha\,\extd_\alpha=0$. 
\begin{eqnarray*}
\extd_\alpha(\extd_\alpha(\xi)) &=& \extd_\alpha(\extd \xi) +(-1)^{n}\,\extd_\alpha\Big(\frac{\partial \xi}{\partial t}\Big)\wedge\extd t + (-1)^n\,\alpha\,\extd_\alpha(\Delta(\xi))\wedge\extd t\cr
&=& \extd^2 \xi+ (-1)^{n+1}\, \frac{\partial\, \extd\xi}{\partial t}\wedge\extd t
+ (-1)^{n+1}\,\alpha\,\Delta(\extd \xi)\wedge\extd t \cr
&&+\, (-1)^{n}\,\extd\Big(\frac{\partial \xi}{\partial t}\Big)\wedge\extd t
 + (-1)^n\,\alpha\,\extd(\Delta(\xi))\wedge\extd t\  .
\end{eqnarray*}
This vanishes for the following reasons. First, $\extd^2=0$ by definition. 
Second, as $\extd^2=0$ we have $\extd\,\Delta=\Delta\,\extd=\extd\,\delta\,\extd$. 
Third, partial $t$ derivative commutes with $\extd$ as they act on different factors of 
$F^n\tens C^\infty(\mathbb{R})$. 

Next check that $\extd_\alpha$ is a signed derivation for the product $\wedge_\alpha$. The only difficult case is the following:
\begin{eqnarray*}
\extd_\alpha(\xi\wedge_\alpha \eta) &=& \extd_\alpha(\xi\wedge \eta)\cr
&&  -\, (-1)^{n+m}\,\alpha\,\extd\big(\delta(\xi\wedge \eta)-
\delta(\xi)\wedge \eta-(-1)^n\,\xi\wedge \delta(\eta)\big)\wedge\extd t\cr
&=& \extd(\xi\wedge \eta)+
(-1)^{n+m}\,\frac{\partial (\xi\wedge\eta)}{\partial t}\wedge\extd t + (-1)^{n+m}\,\alpha\,\Delta(\xi\wedge\eta)\wedge\extd t\cr
&&  -\, (-1)^{n+m}\,\alpha\,\extd\big(\delta(\xi\wedge \eta)-
\delta(\xi)\wedge \eta-(-1)^n\,\xi\wedge \delta(\eta)\big)\wedge\extd t\cr
&=& \extd(\xi\wedge \eta)+
(-1)^{n+m}\,\frac{\partial (\xi\wedge\eta)}{\partial t}\wedge\extd t + (-1)^{n+m}\,\alpha\,\delta\,\extd(\xi\wedge\eta)\wedge\extd t\cr
&&  +\, (-1)^{n+m}\,\alpha\,\extd\big(
\delta(\xi)\wedge \eta+(-1)^n\,\xi\wedge \delta(\eta)\big)\wedge\extd t\ .
\end{eqnarray*}
Now calculate
\begin{eqnarray*}
\extd_\alpha(\xi)\wedge_\alpha\eta &=& 
 \extd \xi\wedge_\alpha\eta +(-1)^{n+m}\,\frac{\partial \xi}{\partial t}\wedge\eta\wedge\extd t + (-1)^{n+m}\,\alpha\,\Delta(\xi)\wedge\eta\wedge\extd t\cr
 &=& \extd \xi\wedge\eta+(-1)^{n+m}\,\alpha\,\big(\delta(\extd\xi\wedge\eta)-\delta\,\extd\xi\wedge\eta
 +(-1)^n\,\extd\xi\wedge\delta\eta
 \big)\wedge\extd t\cr
&&  +\,(-1)^{n+m}\,\frac{\partial \xi}{\partial t}\wedge\eta\wedge\extd t + (-1)^{n+m}\,\alpha\,\Delta(\xi)\wedge\eta\wedge\extd t\cr
 &=& \extd \xi\wedge\eta+(-1)^{n+m}\,\alpha\,\big(\delta(\extd\xi\wedge\eta)+\extd\,\delta\,\xi\wedge\eta
 +(-1)^n\,\extd\xi\wedge\delta\eta
 \big)\wedge\extd t\cr
&&  +\,(-1)^{n+m}\,\frac{\partial \xi}{\partial t}\wedge\eta\wedge\extd t \cr
\xi\wedge_\alpha \extd_\alpha(\eta) &=& \xi\wedge_\alpha \extd\eta +(-1)^m\,\xi\wedge\frac{\partial \eta}{\partial t}\wedge\extd t + (-1)^m\,\alpha\,\xi\wedge\Delta(\eta)\wedge\extd t\cr
&=& \xi\wedge \extd\eta + (-1)^{n+m}\,\alpha\,\big(\delta(\xi\wedge \extd\eta)-
\delta\xi\wedge \extd\eta  -(-1)^n\,\xi\wedge \delta\,\extd\eta\big)\wedge\extd t\cr
&& +\,(-1)^m\,\xi\wedge\frac{\partial \eta}{\partial t}\wedge\extd t + (-1)^m\,\alpha\,\xi\wedge\Delta(\eta)\wedge\extd t\cr
&=& \xi\wedge \extd\eta + (-1)^{n+m}\,\alpha\,\big(\delta(\xi\wedge \extd\eta)-
\delta\xi\wedge \extd\eta  +(-1)^n\,\xi\wedge \extd\,\delta\,\eta\big)\wedge\extd t\cr
&& +\,(-1)^m\,\xi\wedge\frac{\partial \eta}{\partial t}\wedge\extd t \ .
\end{eqnarray*}
Now we can write
\begin{eqnarray*}
\extd_\alpha(\xi\wedge_\alpha \eta) -\extd_\alpha(\xi)\wedge_\alpha\eta-(-1)^n\,\xi\wedge_\alpha \extd_\alpha(\eta) \ =\ 
(-1)^{n+m}\,\alpha\,\mathrm{temp}\wedge\extd t\ ,
\end{eqnarray*}
where
\begin{eqnarray*}
\mathrm{temp} &=& \extd\big(
\delta(\xi)\wedge \eta+(-1)^n\,\xi\wedge \delta(\eta)\big)
 - \big(\extd\,\delta\,\xi\wedge\eta
 +(-1)^n\,\extd\xi\wedge\delta\eta\big) \cr 
&& -\, (-1)^n\, \big(-
\delta\xi\wedge \extd\eta  +(-1)^n\,\xi\wedge \extd\,\delta\,\eta\big)\ ,
\end{eqnarray*}
and this vanishes as $\extd$ is a signed derivation. 

Now we need to show that $\wedge_\alpha$ is an associative product. Again we consider only the difficult case, where $\zeta\in F^p\tens C^\infty(\mathbb{R})$.
\begin{eqnarray*}
&& \zeta\wedge_\alpha(\xi\wedge_\alpha \eta) \cr &=& \zeta\wedge_\alpha\big(
\xi\wedge \eta - (-1)^{n+m}\,\alpha\,\big(\delta(\xi\wedge \eta)-
\delta(\xi)\wedge \eta  -(-1)^n\,\xi\wedge \delta(\eta)\big)\wedge\extd t
\big) \cr
&=& \zeta\wedge_\alpha(
\xi\wedge \eta) \cr &&  -\, (-1)^{n+m}\,\alpha\,\big(\zeta\wedge\delta(\xi\wedge \eta)-
\zeta\wedge\delta(\xi)\wedge \eta  -(-1)^n\,\zeta\wedge\xi\wedge \delta(\eta)\big)\wedge\extd t
\big) \cr
&=& \zeta\wedge
\xi\wedge \eta-(-1)^{p+n+m}\,\alpha\big(\delta(\zeta\wedge
\xi\wedge \eta)-\delta\zeta\wedge\xi\wedge \eta-(-1)^p\,\zeta\wedge\delta(\xi\wedge \eta)
\big)\wedge\extd t
 \cr &&  -\, (-1)^{n+m}\,\alpha\,\big(\zeta\wedge\delta(\xi\wedge \eta)-
\zeta\wedge\delta(\xi)\wedge \eta  -(-1)^n\,\zeta\wedge\xi\wedge \delta(\eta)\big)\wedge\extd t \cr
&=& \zeta\wedge
\xi\wedge \eta \cr
&& -\ (-1)^{p+n+m}\,\alpha\big(\delta(\zeta\wedge
\xi\wedge \eta)-\delta\zeta\wedge\xi\wedge \eta-(-1)^p\,\zeta\wedge\delta\xi\wedge \eta\cr
&& -\ (-1)^{p+n}\,\zeta\wedge\xi\wedge \delta\eta
\big)\wedge\extd t\ ,
\end{eqnarray*}
which is the same as
\begin{eqnarray*}
&& (\zeta\wedge_\alpha\xi)\wedge_\alpha \eta \cr &=&
\big(\zeta\wedge \xi - (-1)^{n+p}\,\alpha\,\big(\delta(\zeta\wedge \xi)-
\delta(\zeta)\wedge \xi  -(-1)^p\zeta\wedge \delta(\xi)\big)\wedge\extd t\big)\wedge_\alpha \eta \cr
&=&
(\zeta\wedge \xi)\wedge_\alpha \eta
 \cr && -\, (-1)^{n+p+m}\,\alpha\,\big(\delta(\zeta\wedge \xi)\wedge \eta-
\delta(\zeta)\wedge \xi\wedge \eta  -(-1)^p\,\zeta\wedge \delta(\xi)\wedge \eta\big)\wedge\extd t \cr
&=&
(\zeta\wedge \xi)\wedge_\alpha \eta
 \cr && -\, (-1)^{n+p+m}\,\alpha\,\big(\delta(\zeta\wedge \xi)\wedge \eta-
\delta(\zeta)\wedge \xi\wedge \eta  -(-1)^p\,\zeta\wedge \delta(\xi)\wedge \eta\big)\wedge\extd t \cr
&=&
(\zeta\wedge \xi)\wedge \eta - (-1)^{n+p+m}\,\alpha\,\big(
\delta(\zeta\wedge \xi\wedge \eta)-\delta(\zeta\wedge \xi)\wedge \eta-(-1)^{p+n}\,
\zeta\wedge \xi\wedge \delta\eta
\big)\wedge\extd t
 \cr && -\, (-1)^{n+p+m}\,\alpha\,\big(\delta(\zeta\wedge \xi)\wedge \eta-
\delta(\zeta)\wedge \xi\wedge \eta  -(-1)^p\,\zeta\wedge \delta(\xi)\wedge \eta\big)\wedge\extd t \ .\quad\blacksquare
\end{eqnarray*}

\section{Examples of homotopies}

\subsection{A homotopy giving diffusion}  \label{cbhajksvckvkvv1}
Start with a differential manifold $M$ with local coordinates $x^\mu$. 
We take 
\begin{eqnarray}  \label{cvdgsyjhtc}
\delta(\xi) &=& g^{\mu\nu}\,\frac{\partial}{\partial x^\mu}\, \righthalfcup\,\nabla_\nu (\xi)\ .
\end{eqnarray}
Here $\nabla_\nu$ is a covariant derivative on the manifold, and the symbol $\righthalfcup$ denotes 
the interior product of a vector field and an $n$-form, resulting in an $n-1$ form. For example, we have
\begin{eqnarray}
\frac{\partial}{\partial x^1}\, \righthalfcup\, (\extd x^1\wedge\extd x^2)\ =\ \extd x^2\ ,\quad
\frac{\partial}{\partial x^2}\, \righthalfcup\, (\extd x^1\wedge\extd x^2)\ =\ -\,\extd x^1\ .
\end{eqnarray}
In general, to find $\frac{\partial}{\partial x^i}\, \righthalfcup\, \xi$, 
use a permutation on $\xi$ to put $\extd x^i$ to the front, multiplying by its sign, then cancel the $\frac{\partial}{\partial x^i}$ on the front. The result is zero if there is no $\extd x^i$ in $\xi$.

Remembering that $\Omega^{-1}M=0$, we have the following formula for $\Delta$ on functions:
\begin{eqnarray}\label{cvhskvcxzdh}
\Delta(f) &=& \delta(\extd f)\ =\ \delta\big(\frac{\partial f}{\partial x^\kappa}\,\extd x^\kappa\big) \cr
&=& g^{\mu\nu}\,\frac{\partial}{\partial x^\mu}\, \righthalfcup\,\nabla_\nu \big(\frac{\partial f}{\partial x^\kappa}\,\extd x^\kappa\big) \cr
&=&  g^{\mu\nu}\,\frac{\partial}{\partial x^\mu}\, \righthalfcup\,\big(
\frac{\partial^2 f}{\partial x^\nu\,\partial x^\kappa}\,\extd x^\kappa+
\frac{\partial f}{\partial x^\kappa}\,\nabla_\nu (\extd x^\kappa)\big) \cr
&=&  g^{\mu\nu}\,\frac{\partial}{\partial x^\mu}\, \righthalfcup\,\big(
\frac{\partial^2 f}{\partial x^\nu\,\partial x^\kappa}\,\extd x^\kappa-
\frac{\partial f}{\partial x^\kappa}\,\Gamma^\kappa_{\nu\lambda}\,\extd x^\lambda\big) \cr
&=&  g^{\mu\nu}\, \frac{\partial^2 f}{\partial x^\nu\,\partial x^\mu} - 
g^{\mu\nu}\, \frac{\partial f}{\partial x^\kappa}\,\Gamma^\kappa_{\nu\mu}\ .
\end{eqnarray}
Here we have used the usual Christoffel symbols $\Gamma^\kappa_{\nu\mu}$ for a covariant derivative on a tangent or cotangent bundle. If we take $g^{\mu\nu}$ to be a Riemannian metric
(i.e.\ a non-degenerate symmetric positive matrix in the given coordinate system) and 
$\nabla_\nu$ the associated Levi-Civita connection, then $\Delta$ is the 
Laplace operator. Remember that for the Levi-Civita connection,
\begin{eqnarray}
\Gamma^\kappa_{\nu\mu} &=& \frac12 \ g^{\kappa\lambda}\,\big(
\frac{\partial g_{\nu\lambda}}{\partial x^\mu} + \frac{\partial g_{\lambda\mu}}{\partial x^\nu} - \frac{\partial g_{\nu\mu}}{\partial x^\lambda}
\big)\ .
\end{eqnarray}
We compare (\ref{cvhskvcxzdh}) to the usual form of the Laplace operator,
\begin{eqnarray}
\Delta f &=& \frac1{\sqrt{|g|}}\ \frac{\partial }{\partial x^\nu}\Big(\sqrt{|g|}\,g^{\mu\nu}\,
\frac{\partial f}{\partial x^\mu} \Big)\cr
&=& g^{\mu\nu}\, \frac{\partial^2 f}{\partial x^\nu\,\partial x^\mu} + \Big(
\frac12\ g^{\mu\nu}\ \frac{\partial \log(|g|)}{\partial x^\nu}+ \frac{\partial g^{\mu\nu}}{\partial x^\nu}
\Big)\frac{\partial f}{\partial x^\mu}\cr
&=& g^{\mu\nu}\, \frac{\partial^2 f}{\partial x^\nu\,\partial x^\mu} + \Big(
\frac12\ g^{\kappa\lambda}\ \frac{\partial \log(|g|)}{\partial x^\lambda}+ \frac{\partial g^{\kappa\nu}}{\partial x^\nu}
\Big)\frac{\partial f}{\partial x^\kappa}
\end{eqnarray}
where $|g|$ is the determinant of the matrix $g_{\nu\mu}$. Now we use the following formulae for an invertible matrix valued function $A(x)$ of a variable $x$,
\begin{eqnarray}
\frac{\extd \log\det A}{\extd x} \ =\ \mathrm{trace}\Big(A^{-1}\,\frac{\extd A}{\extd x}\Big)\ ,\quad 
\frac{\extd (A^{-1})}{\extd x}\ =\ -\, A^{-1}\,  \frac{\extd A}{\extd x}    \, A^{-1}\ .
\end{eqnarray}
From these we can verify that the $\Delta$ in the deformed algebra (see (\ref{cvhskvcxzdh})) is indeed the Laplace operator applied to functions, as follows:
\begin{eqnarray}
\frac12\ g^{\kappa\lambda}\ \frac{\partial \log(|g|)}{\partial x^\lambda}+ \frac{\partial g^{\kappa\nu}}{\partial x^\nu} &=& 
\frac12\ g^{\kappa\lambda}\, g^{\mu\nu}\, \frac{\partial g_{\nu\mu}}{\partial x^\lambda} -
g^{\kappa\lambda}\,\frac{\partial g_{\lambda\mu}}{\partial x^\nu}\,g^{\mu\nu} \cr
&=& \frac12\ g^{\kappa\lambda}\ \Big(
g^{\mu\nu}\, \frac{\partial g_{\nu\mu}}{\partial x^\lambda} -
\frac{\partial g_{\lambda\mu}}{\partial x^\nu}\,g^{\mu\nu}
-\frac{\partial g_{\lambda\nu}}{\partial x^\mu}\,g^{\nu\mu}\Big) \cr
&=& -\, g^{\mu\nu}\, \Gamma^\kappa_{\nu\mu}\ .
\end{eqnarray}

\subsection{A homotopy giving drift}  \label{cbhajksvckvkvv2}
Take a vector field $ v = \sum\limits_{a}v^{a}\,\frac{\partial}{\partial x^{a}} $ and write, for a form $\xi$
\begin{equation} \label{vucgkklycf}
\delta_v(\xi)= v \,\righthalfcup\, \xi = v^{a}\,\big(\frac{\partial}{\partial x^{a}}\righthalfcup\, \xi\big).
\end{equation}
Now for a function $ f $ 
\begin{align*}
\Delta_v(f)= \delta df =v^{a}\big(\frac{\partial}{\partial x^{a}}\,\righthalfcup \, \extd x^{b}\cdot\frac{\partial f}{\partial x^{b}}\big)=v^{a}\,\frac{\partial f}{\partial x^{a}}\ ,
\end{align*}
which is the derivative of $ f $ in the direction of the vector field $ v $. For a $ 1 $-form
\begin{align*}
\Delta(\eta)&=\Delta\big(\eta_{b}\,\extd x^{b}\big)\,=\,
\extd\delta \big (\eta_{b}\,\extd x^{b}\big)+ \delta \extd \big(\eta_{b}\,\extd x^{b}\big)\\
&=\extd \big (v^{a}\eta_{a}\big)+\delta\big(\frac{\partial \eta_{b}}{\partial x^{c}}\,\extd x^{c}\wedge \extd x^{b}\big)\\
&= \frac{\partial \big (v^{a}\eta_{a}\big) }{\partial x^{b}}\,\extd x^{b}+v^{a}\,\frac{\partial}{\partial x^{a}}\,\righthalfcup\,\big(\frac{\partial \eta_{b}}{\partial x^{c}}\,\extd x^{c}\wedge \extd x^{b}\big)\\
&= \frac{\partial \big (v^{a}\eta_{a}\big) }{\partial x^{b}}\,\extd x^{b}+v^{a}\,\frac{\partial \eta_{b}}{\partial x^{a}}\,\extd x^{b}- v^{a}\,\frac{\partial \eta_{a}}{\partial x^{b}}\,\extd x^{b}\\
&= \frac{\partial v^{a} }{\partial x^{b}}\,\eta_{a}\,\extd x^{b}+v^{a}\,\frac{\partial \eta_{b}}{\partial x^{a}}\,\extd x^{b}.
\end{align*} 
This is the Lie derivative of the $ 1 $-form $ \eta $ along the vector field $ v $. 

\section{An isomorphism of differential graded algebras} \label{cvsaucxrx}
There is a map $\mathcal{I}:F_{\mathbb{R}}^n\to F_{\mathbb{R}}^n$ given by (for $\xi\in F^n\tens C^\infty(\mathbb{R})$)
\begin{eqnarray}
\mathcal{I}(\xi) &=& \xi-(-1)^n\,\alpha\,\delta(\xi)\wedge \extd t\ ,\cr
\mathcal{I}(\xi\wedge\extd t) &=& \xi\wedge\extd t\ ,
\end{eqnarray}
with the property that $\mathcal{I}$ gives an isomorphism from 
$(F_{\mathbb{R}}^*,\extd_0,\wedge_0)$ to $(F_{\mathbb{R}}^*,\extd_\alpha,\wedge_\alpha)$. The two most difficult parts to check are, for $\eta\in F^m\tens C^\infty(\mathbb{R})$,
\begin{eqnarray}
\mathcal{I}(\xi)\wedge_\alpha \mathcal{I}(\eta) &=& \xi  \wedge_\alpha \eta - (-1)^m\,\alpha\,\xi\wedge\delta(\eta)\wedge\extd t - (-1)^n\,\alpha\,\delta(\xi)\wedge\extd t\wedge\eta \cr
&=& \xi  \wedge_\alpha \eta - (-1)^m\,\alpha\,\xi\wedge\delta(\eta)\wedge\extd t - (-1)^{n+m}\,\alpha\,\delta(\xi)\wedge\eta\wedge\extd t \cr
&=&  \xi  \wedge_0 \eta - (-1)^{n+m}\,\alpha\,\delta(\xi\wedge_0\eta)\wedge\extd t \cr
&=& \mathcal{I}(\xi\wedge_0\eta)\ ,
\end{eqnarray}
and
\begin{eqnarray}
\extd_\alpha\,\mathcal{I}(\xi) &=& \extd_\alpha\xi-(-1)^n\,\alpha\,\extd_\alpha(\delta(\xi)\wedge\extd t)
\cr
&=& \extd_0\xi+(-1)^n\,\alpha\,\Delta(\xi)\wedge\extd t-(-1)^n\,\alpha\,\extd\delta(\xi)\wedge\extd t\cr
&=& \extd_0\xi+(-1)^n\,\alpha\,\delta\extd(\xi)\wedge\extd t\cr
&=& \mathcal{I}(\extd_0\xi)\ .
\end{eqnarray}
It is easy to see that the inverse of $\mathcal{I}$ is given by the formulae
\begin{eqnarray}
\mathcal{I}^{-1}(\xi) &=& \xi+(-1)^n\,\alpha\,\delta(\xi)\wedge \extd t\ ,\cr
\mathcal{I}^{-1}(\xi\wedge\extd t) &=& \xi\wedge\extd t\ .
\end{eqnarray}

As we shall see in Section~\ref{vcaisuvfiuyt}, some results obtained by using the map $\mathcal{I}$ look rather strange. The reason is quite simple: $\mathcal{I}$ changes the module structure, so the definition of multiplying (for example) functions by forms changes.

The cohomology of the DGA $(F_{\mathbb{R}}^*,\extd_\alpha,\wedge_\alpha)$ is defined to be
\begin{eqnarray}
H^n(F_{\mathbb{R}}^*,\extd_\alpha)\,=\,\frac{\mathrm{kernel}\, \extd_\alpha:F_{\mathbb{R}}^n\to F_{\mathbb{R}}^{n+1}}
{\mathrm{image}\, \extd_\alpha:F_{\mathbb{R}}^{n-1}\to F_{\mathbb{R}}^{n}}\ .
\end{eqnarray}
 The wedge product of forms gives a product on the cohomology.
If $F_{\mathbb{R}}^*=\Omega^*(M\times\mathbb{R})$ with the usual differential, then 
$H^n(F_{\mathbb{R}}^*,\extd_0)$ is the de Rham cohomology
$H^n_{dR}(M\times\mathbb{R})$.

\begin{cor} \label{vcyudksyufvj}
The cohomology of the DGA $(F_{\mathbb{R}}^*,\extd_0,\wedge_0)$ is the same as that of
$(F_{\mathbb{R}}^*,\extd_\alpha,\wedge_\alpha)$.
\end{cor}
\noindent {\bf Proof:}\quad Use the isomorphism in this section.\quad$\blacksquare$

\medskip
 A way to use this result would be to say that if $H^n(F_{\mathbb{R}}^*,\extd_0)=0$, then also $H^n(F_{\mathbb{R}}^*,\extd_\alpha)=0$, so 
 \begin{eqnarray}
{\mathrm{kernel}\, \extd_\alpha:F_{\mathbb{R}}^n\to F_{\mathbb{R}}^{n+1}} \,=\, 
{\mathrm{image}\, \extd_\alpha:F_{\mathbb{R}}^{n-1}\to F_{\mathbb{R}}^{n}}\ .
\end{eqnarray}
We shall use this in Section~\ref{bcuyowiiygccjbehb}. 
However there is a caveat to this application: The de Rham cohomology is defined using smooth functions, and the standard results refer to that case. More care needs to be taken when using functions which are only finitely differentiable. We now look in a little more detail at the kernel of $\extd_\alpha$. 

\begin{prop}\label{prop6}
Suppose $ \xi \in F^{1}\tens C^{\infty} (\R)$ and $ a \in F^{0}\tens C^{\infty} (\R)$. Then:

1)\quad 
$ \extd _{\alpha}(\xi+a .\extd t)=0 $ if and only if $ \extd \xi =0 $ and $ \extd (a-\alpha\,\delta\,\xi)=\frac{\partial \xi}{\partial t} $.

2)\quad If $ \extd _{\alpha}(\xi+a .\extd t)=0 $ and for some $b \in F^{0}\tens C^{\infty} (\R)$ we have $\xi=\extd b$,  then $\extd(a-\alpha\,\delta\extd b-\frac{\partial b}{\partial t})=0$. 

3) Under the conditions for (2), if $F$ is connected (i.e.\ the kernel of $\extd:F^0\to F^1$ consists of constants times the identity) then there is $c \in F^{0}\tens C^{\infty} (\R)$ so that $\xi=\extd c$ and 
 $a-\alpha\,\Delta c-\frac{\partial c}{\partial t}=0$.

\end{prop}

\begin{proof}
We have
\begin{eqnarray*}
\extd _{\alpha}(\xi+a \,\extd t)&=& \extd \xi -\dfrac{\partial \xi}{\partial t} \wedge \extd t-\alpha\,\Delta(\xi) \wedge \extd t +\extd a \wedge \extd t,
\end{eqnarray*}
and if this is zero, we have $ \extd \xi =0 $ and $ (\extd a -\alpha\,\Delta(\xi)- \frac{\partial \xi}{\partial t})\wedge \extd t =0 $ so $ \extd \,a = \alpha\,\Delta(\xi)+ \frac{\partial \xi}{\partial t} $, and we have $ \Delta(\xi)=\delta\,\extd \xi+ \extd \delta\,\xi $, if $ \extd \xi =0 $, then $ \Delta(\xi)= \extd \delta\,\xi $, so  $ \extd a = \alpha\,\extd \delta\,\xi+ \frac{\partial \xi}{\partial t} $, and this is $ \extd (a - \alpha\, \delta\,\xi)= \frac{\partial \xi}{\partial t}$. For the second part, $\extd \frac{\partial b}{\partial t}=\frac{\partial \extd b}{\partial t}$. For the third part, the connectedness assumption gives $a-\alpha\,\delta\extd b-\frac{\partial b}{\partial t}=f(t)$, and we correct $b$ to give $c$ using the integral of $f(t)$.
\end{proof}

\section{It\^o and Stratonovich calculus}\label{kjychxjy}

\subsection{Homotopy deformation by diffusion and drift}  \label{vkujkhtdcty}
Use a linear combination of the $ \delta $ from Sections~\ref{cbhajksvckvkvv1} and \ref{cbhajksvckvkvv2}, with two parameters $ \alpha $ and $ \beta $ instead of just one as previously, and nothing else is affected.
We call $ \delta $ and $ \Delta $ from Section~\ref{cbhajksvckvkvv1} $ \delta_{\text{diff}} $ and $ \Delta_{\mathrm{diff}} $
(remember that $ \Delta_{\mathrm{diff}} $ is the usual Laplace operator), and from Section~\ref{cbhajksvckvkvv2} we use $ \delta_{\text{v}} $ and $ \Delta_{\text{v}} $.
Then we get 
\begin{eqnarray} \label{bvhdjlslv}
\extd_{\alpha\, \beta}\xi&=& \extd\xi + (-1)^{n}\,\frac{\partial \xi}{\partial t}\wedge \extd t+ (-1)^{n}\,\alpha \Delta_{\mathrm{diff}}(\xi)\wedge \extd t+  (-1)^{n}\beta \Delta_{\text{v}}(\xi)\wedge \extd t,\cr
\xi \wedge_{\alpha\beta}\eta &=& \xi\wedge\eta\cr
&&-(-1)^{n+m}\alpha \big(\delta_{\text{diff}}(\xi\wedge\eta)-\delta_{\text{diff}}(\xi)\wedge\eta-\,(-1)^{n}\,\xi\wedge \delta_{\text{diff}}(\eta)) \big)\wedge \extd t \cr 
&& -\,(-1)^{n+m}\beta \big(\delta_{\text{v}}(\xi\wedge\eta)-\delta_{\text{v}}(\xi)\wedge\eta-\,(-1)^{n}\xi\wedge \delta_{\text{v}}(\eta)) \big)\wedge \extd t.
\end{eqnarray}

In the case of function $f$ we have 
\begin{equation}
\extd_{\alpha\beta} f=\extd f +\frac{\partial f}{\partial t}+ \alpha \,\Delta_{\mathrm{diff}}\,f	\,\extd t+\beta\, v^{a}\frac{\partial  f}{\partial v^{a}}\, \extd t
\end{equation}
The product of a function and a $ 1 $-form $\eta$ is modified by
\begin{equation}\label{eq 6}
f\wedge_{\alpha \beta}\,\eta= f\,\eta+\alpha\big(\delta_{\text{diff}} (f\,\eta)- f .\delta_{\text{diff}}(\eta)\big)\wedge \extd t+\beta\big(\delta_{\text{v}}(f\,\eta)-f .\delta_{\text{v}}(\eta) \big)\wedge \extd t.
\end{equation}
Now $ \delta_{\text{v}}(f\eta)= f \cdot \delta_{\text{v}}\,\eta $ so the $ \beta $ term in (\ref{eq 6}) vanishes. However the $ \alpha $ term has 
\begin{align*}
\delta_{\text{diff}} (f\,\eta)- f \cdot\delta_{\text{diff}}\,(\eta)&=g^{\mu\nu}\,\frac{\partial}{\partial x^{\mu}}\,\righthalfcup\,\big(\nabla_{\nu} \,(f\,\eta)-f \cdot \nabla_{\nu}\,(\eta)\big)\\
&=g^{\mu\nu}\,\frac{\partial}{\partial x^{\mu}}\,\righthalfcup\,\big(\frac{\partial f}{\partial x^{\nu}} \, \eta\big)\\
&=\mathrm{grad}(f)\,\righthalfcup\, \eta
\end{align*}
 where grad is the usual gradient of a function, so the product of functions and 1-forms is deformed to
\begin{equation} \label{jkhgcxfjmngc}
\eta \wedge_{\alpha \beta}\,f \,=\, f\wedge_{\alpha \beta}\,\eta \,=\, f\,\eta+\alpha\, \big(\mathrm{grad}(f)\,\righthalfcup\, \eta\big)\, \extd t\ .
\end{equation}
In particular we consider the case $\alpha = 1/2$ and $\beta = 1$. We denote the corresponding  deformed derivative and product by $\di _I$ and $\wedge_I$.



\subsection{Brownian motion on the real line} \label{bcuopagvuiou}
We now consider It\^{o} calculus based on a real valued Brownian motion. 
Briefly (for the completely uninitiated and skipping all of the interesting details) It\^{o} calculus for diffusions is built on the construction of integrals with respect to a Brownian motion $B_t$. The aim is to construct integrals of the form $\int_0^t f(B_s,s) \di B_s$ for smooth functions $f$. The problem is that $B$ is nowhere differentiable, indeed it is of unbounded variation, but finite quadratic variation. As a consequence if one tries to construct the integral as a Riemann sum then the value of the integral depends on the point at which the function is sampled within the partition intervals. There are two conventions in common use, the It\^{o} and Stratonovich integrals defined respectively as suitable limits of Riemann sums of the form,
\begin{align*}
\int_0^t f(B_s,s) \di B_s& = \lim \sum f(B_{t_i},t_i)(B_{t_{i+1}}-B_{t_i}),\\
\int_0^t f(B_s,s) \circ \partial B_s & = \lim\sum \frac{1}{2}(f(B_{t_i})+f(B_{t_{i+1}})) (B_{t_{i+1}}-B_{t_i}).
\end{align*}
Using such integrals it is possible to define processes by writing stochastic differential equations. Indeed there is a chain rule for each form of integral - the It\^{o} formula in the It\^{o} case and the normal chain rule in the Stratonovich case.

As a simple example we first take $B_t$ to be Brownian motion on the real line, with time $t$. For what follows we define a process $X_t$ by the It\^{o} stochastic differential equation,
\[\di X_t =  \di B_t .\]
In stochastic differential equations this is merely short hand for the stochastic integral equation,
\[\int_0^t \di X_s = \int_0^t \di B_s,\]
which has the obvious trivial solution $X_t = B_t$. (We will consider more general processes $X_t$ later.)
Then for any smooth function $f= f(x,t)$ on $\mathbb{R}^2$ It\^{o}'s formula gives,
\begin{equation}\label{ito1}
\di f(X_t,t) =\left( \frac{\partial f}{\partial t}(X_t,t) + \frac{1}{2}\frac{\partial^2 f}{\partial x^2} (X_t,t)\right)  \di t + \frac{\partial f}{\partial x}(X_t,t) \di X_t. 
\end{equation}
(note the additional second order term which would not be in the usual chain rule).
If we denote the Stratonovich integral by $\circ\,\partial X_t$ then the It\^{o} and Stratonovich integrals are related to each other by,
\begin{eqnarray} \label{kucfgku}
f(X_t,t) \circ \partial X_t 
=f(X_t,t) \di X_t + \frac{1}{2} \frac{\partial f}{\partial x}(X_t,t)\di t.
\end{eqnarray}

Now consider the trivial example of our DGA given by a Riemannian manifold $M = \mathbb{R}$ with the usual metric together with the derivative $\di_{I}$ from Subsection~\ref{vkujkhtdcty} where $v=0$. 
We first note that our choice of homotopy $\delta = \frac{1}{2} \delta_{\mathrm{diff}} + \delta_{\mathrm{v}}$ gives,
\[\delta \mathrm{d} + \mathrm{d} \delta = \frac{1}{2}\Delta_M\]
where $\Delta_M$ denotes the Laplace operator on the manifold $M$. 

Now we consider the deformed derivative $\di_{I}$ acting on some smooth function $f(x,t)$ defined on $M\times \mathbb{R}$,
\begin{eqnarray}
\nonumber \extd_{I}f&=& \frac{\partial f}{\partial x}\,\extd x+ \frac{\partial f}{\partial t}\, \extd t+ \frac12\, \frac{\partial^2 f}{\partial x^{2}}\, \extd t\\
&=& \frac{\partial f}{\partial x}\,\extd_{I} x+ \left(\frac{\partial f}{\partial t}+ \frac12\, \frac{\partial^2 f}{\partial x^{2}}\right) \,  \extd_{I} t,\label{ito2}
\end{eqnarray}
where we have $\di x = \di_{I} x$ and $\di t = \di_{I} t$. 
Comparing \eqref{ito2} to the It\^{o} formula \eqref{ito1} we see that  the operator $\di_I$ is equivalent to the It\^{o} differential i.e. everywhere we have $\di$ in our It\^{o} calculus we want to have $\di_{I}$.  Moreover we substitute $x = X_t$ always; that is the calculus $\di_I$ gives the same answers as if we move on the manifold $M$ according to the diffusion $X_t$. 
Moreover,
\begin{eqnarray}  \label{kucfgku1}
f \wedge_{I} \di_{I} x = f \di_{I} x + \frac{1}{2} \frac{\partial f}{\partial x} \di_{I} t
\end{eqnarray}
Thus we can see that  the natural interpretation is for $\di_{I}$ to be the It\^{o} differential  whilst $f\wedge_{I} \di_{I} x$ denotes the Stratonovich integral.

\subsection{A general diffusion on $\mathbb{R}^n$}\label{RN}
Consider a diffusion $X_t$ in $\mathbb{R}^n$ defined by
\[\di X_t^i = v^i(X_t) \di t + \sigma^{ij}(X_t) \di B_t^j\]
where $B_t$ is an $\mathbb{R}^n$ valued Brownian motion.
Then for any smooth function $f$ on $\mathbb{R}^{n+1}$ It\^{o}'s formula gives,
\begin{align*}
\di f(X_t,t) &=\left( \frac{\partial f}{\partial t}  +v^i(X_t) \frac{\partial f}{\partial x^i}  + \frac{1}{2}a_{ij}(X_t)\frac{\partial^2 f}{\partial x^i\partial x^j}  \right)\di t  
+ \sigma_{ij}(X_t)\frac{\partial f}{\partial x^i} \di B_t^j \\
&=\left( \frac{\partial f}{\partial t}(X_t,t)  + \frac{1}{2}a_{ij}(X_t)\frac{\partial^2 f}{\partial x^i\partial x^j}(X_t,t)  \right)\di t  
+\frac{\partial f}{\partial x^i}(X_t,t)\di X_t^i 
\end{align*}
where $a\, =\, \sigma \sigma^T$.
We also recall the It\^{o} product formula,
\begin{align*}
\di (f(X_t) h(X_t)) &= f(X_t) \di h(X_t) + h(X_t) \di f(X_t) + \di f(X_t) \di h(X_t)\\
& =  f(X_t) \di h(X_t) + h(X_t) \di f(X_t)+g^{ij}(X_t)\frac{\partial f}{\partial x^i}\frac{\partial h}{\partial x^j}\di t
\end{align*}
and note the relation between It\^{o} and Stratonovich integrals,
\[f^i \circ\partial X^i_t = f^i \di X_t^i + \frac{1}{2} \frac{\partial f^i}{\partial x^i} \di t\]

Now for our DGA we take $M = \mathbb{R}^n$ with the Riemannian metric $g_{ij} = (a^{-1})_{ij}$.
Again we have $\delta = \frac{1}{2} \delta_{\mathrm{diff}} + \delta_{\mathrm{v}}$ where now $v\not\equiv 0$. This gives,
\[\delta \di + \di \delta = \frac{1}{2} \Delta_M + v^i \frac{\partial}{\partial x^i} =\frac{1}{2} g^{ij}\frac{\partial^2 }{\partial x^i\partial x^j}- \frac{1}{2}g^{jk}\Gamma_{jk}^i\frac{\partial }{\partial x^i}
+ v^i \frac{\partial}{\partial x^i}.\]


Thus for a smooth function $f$ defined on $M\times \mathbb{R}$ we have,
\begin{align*}
\extd_{I} f&=\extd f +\left(\frac{\partial f}{\partial t}+ \frac{1}{2} \,\Delta_{M}\,f+ v^{i}\frac{\partial  f}{\partial x^{i}}\right)\di t\\
&=\frac{\partial f}{\partial x^i} \di x^i+\left(\frac{\partial f}{\partial t}+ \frac{1}{2} g^{ij}\frac{\partial^2 f}{\partial x^i\partial x^j} -\frac{1}{2} g^{jk}\Gamma_{jk}^i\frac{\partial f}{\partial x^i}+ v^{i}\frac{\partial  f}{\partial x^{i}}\right)\di t
\end{align*}
Again by applying this to $f(x) = x^l$ we have,
\[\di_{I} x^l =\di x^l+\left(v^l-\frac{1}{2} g^{jk}\Gamma_{jk}^l\right)\di t\]
and so,
\begin{align*}
\extd_{I} f
&=\frac{\partial f}{\partial x^i} \di_{I} x^i+\left(\frac{\partial f}{\partial t}+ \frac{1}{2} a_{ij} \frac{\partial^2 f}{\partial x^i \partial x^j}\right)\di t
\end{align*}

Also we note that,
\begin{align*}
f^i \wedge_I \di_I x^i
&= f^i \wedge_I \di x^l+\left(f^i \wedge \left(v^l-\frac{1}{2} g^{jk}\Gamma_{jk}^l\right)\right)\wedge\di t \\
&= f^i  \di x^i+\frac{1}{2} \frac{\partial f^i}{\partial x^i}\di t+f^i  \left(v^l-\frac{1}{2} g^{jk}\Gamma_{jk}^l\right)\di t \\
&= f^i  \di_I x^i+\frac{1}{2} \frac{\partial f^i}{\partial x^i}\di t,
\end{align*}
which can be again compared with the Stratonovich integral.

We can also consider the differential of a product of functions $f,h$:
\begin{eqnarray}
\extd_{I}(f\,h) &=& f\wedge_{I}\extd_{I} h+\extd_{I} f\wedge_{I} h 
\end{eqnarray}
and
\begin{align*}
f \wedge_I \di_I h& = 
f  \wedge_I \frac{\partial h}{\partial x^i} \di_I x^i+f\wedge_I \left( \frac{1}{2} g^{ij}\frac{\partial^2 h}{\partial x^i\partial x^j} \right)\di t\\
& = f  \frac{\partial h}{\partial x^i} \di_I x^i 
+ \frac{1}{2} g^{ik} \frac{\partial f}{\partial x^i}   \frac{\partial h}{\partial x^k}\di t +f \left( \frac{1}{2} g^{ij}\frac{\partial^2 h}{\partial x^i\partial x^j} \right)\di t\\
& = f  \di_I h
+ \frac{1}{2} g^{ik} \frac{\partial f}{\partial x^i}   \frac{\partial h}{\partial x^k}\di t.\end{align*}
Thus we have,
\begin{eqnarray}
\extd_{I}(f\,h) &=&f  \di_I h+ h \di_I f
+ g^{ik} \frac{\partial f}{\partial x^i}   \frac{\partial h}{\partial x^k}\di t.
\end{eqnarray}


%

\subsection{The It\^o-Stratonovich calculus}
As set out in Subsections ~\ref{bcuopagvuiou} and \ref{RN}, the differential $\extd_{I} $ (with $\alpha=\frac12$ and $\beta=1$) corresponds to the It\^o differential. However,
the differential graded algebra with $\extd_{I}$ and $\wedge_{I}$ actually corresponds to the Stratonovich calculus. 
The equation (\ref{kucfgku}) relating the It\^o and Stratonovich calculi corresponds to (\ref{kucfgku1}), and there
the RHS with ordinary multiplication and $\extd_{I}$ contains the It\^o terms and the LHS with 
$\extd_{I}$ and $\wedge_{I}$ contains the Stratonovich term. The same holds in the example in  Subsection \ref{RN}. We shall call the DGA with $\extd_{I}$ and $\wedge_{I}$ the It\^o-Stratonovich DGA. 

By Section~\ref{cvsaucxrx} the DGA with $\extd_{I}$ and $\wedge_{I}$ is isomorphic to the classical calculus, with 
$\extd_{0}$ and $\wedge_{0}$. This simply reflects the well known fact (in first order) that formulae involving the Stratonovich calculus are of the same form as the formulae of the classical differential calculus.

We note that as things stand  we have simply demonstrated it is possible to construct a deformation of a  differential graded algebra which happens to coincide with the formulas for It\^{o} calculus. However in the next section we hope to convince the reader there are interesting applications for this approach.

\subsection{An Application} \label{bcuyowiiygccjbehb}

Unlike in stochastic analysis we have a graded algebra consisting of forms of all orders. We now look at an application which takes advantage of this and the concept of cohomology to prove a recent result from stochastic analysis.

 Consider a diffusion,
\[\di X_t^i = -v^i(X_t) \di t + \sigma^{ij}(X_t) \di B_t^j\]
on a probability space $(\Omega,\mathcal{F},\mathbb{P})$.
In a recent paper \cite{Truman}, conditions were derived using It\^{o} calculus under which the Girsanov change of measure,
\[\exp\left(\int_0^t\langle \sigma^{-1}v ,\di B_s\rangle  - \frac{1}{2} \int_0^t\| \sigma^{-1}v\|^2 \di s \right)\]
is independent of the path of the process $X_s$ for $s\in[0,t]$.
That is, it was shown that there exists a function $f$ such that
\begin{equation}\label{JLWn}
\di f(X_t,t) =\langle \sigma^{-1}v ,\di B_t\rangle  - \frac{1}{2} \| \sigma^{-1}v\|^2 \di t 
\end{equation}
if and  only if $f, \sigma, v$ satisfy the PDEs,
\begin{equation}\label{iffna}
\left\{\begin{array}{rl}
\displaystyle
 \frac{1}{2} (\sigma \sigma^T)^{ij}\frac{\partial^2 f}{\partial x^i\partial x^j}
& = \displaystyle 
 \frac{1}{2}\frac{\partial f}{\partial x^i}v^i-\frac{\partial f}{\partial t}
\\[1em]
\displaystyle
\frac{\partial f}{\partial x^i} &=\displaystyle  ((\sigma\sigma^T)^{-1})^{ij}v^j
 \end{array}\right.
\end{equation}

For simplicity we define the matrix $a = \sigma \sigma^T$ and note that \eqref{iffna}  implies,
(differentiating the first equation with respect to $x^l$ and eliminating $f$),
\begin{align}
 \nonumber \frac{1}{2} a^{ij}\frac{\partial^2 }{\partial x^i\partial x^j}\left((a^{-1})^{lk}v^k\right)
 & = \displaystyle -\frac{1}{2} \frac{\partial a^{ij}}{\partial x^l}\frac{\partial}{\partial x^i}\left((a^{-1})^{jk}v^k\right)
+ \frac{1}{2}(a^{-1})^{ij}v^j\frac{\partial v^i}{\partial x^l}\\
&\qquad +\frac{1}{2}v^i\frac{\partial }{\partial x^i}\left((a^{-1})^{lj}v^j\right)-\frac{\partial }{\partial t}\left((a^{-1})^{lj}v^j\right)\label{differentiated}
\end{align}

We now show that this can be derived by considering a second order differential form in our It\^{o}-Stratonovich calculus.
 We consider the manifold $M = \mathbb{R}^n$ from Section \ref{RN} but with $v\mapsto -v$.
 
%


We note that the equation \eqref{JLWn} can be written in It\^{o} calculus in the form,
\[
\di f(X_t,t) 
= 
(a^{-1})^{ij}v^j\di X_t^i
 + \frac{1}{2}  (a^{-1})^{ij}v^i v^j\di t 
\]
Thus the equivalent question in our DGA is to ascertain when is the form,
\[\xi = g_{ij}v^j\di_{I}x^i
 + \frac{1}{2}  g_{ij}v^i v^j\di t \]
exact with respect to $\di_{I}$ where we recall that the Riemannian metric is give by $g_{ij} = (a^{-1})_{ij}$. Thus we consider when,
\[\di_{I}\xi =0.\]
(See the discussion after Corollary~\ref{vcyudksyufvj} on this point.) Note that here we are considering a second order form which has no corresponding concept in stochastic anlaysis. We proceed using the rules for the It\^{o}-Stratonovich DGA.
Since,
\[\di_{I} x^i =\di x^i-\left(v^i+\frac{1}{2} g^{lm}\Gamma_{lm}^i\right)\di t\]
we have,
\begin{align} 
\di_I \xi &= \extd\left(g_{ij}v^j\di x^i\right) 
-\frac{\partial }{\partial t}\left(g_{ij}v^j\right)\di x^i \wedge \extd t
-\frac12\, \Delta_{\mathrm{diff}}\left(g_{ij}v^j\di x^i\right)\wedge \extd t\nonumber\\
&\qquad- \Delta_{-v}\left(g_{ij}v^j\di x^i \right)\wedge \extd t
-\frac12\frac{\partial}{\partial x^k}\left(g_{ij}v^jv^i+g_{ij} g^{lm}\Gamma_{lm}^iv^j\right)\di x^k\wedge\di t.\label{closed}
\end{align}

Moreover,
\[ \Delta_{-{v}}\left(g_{ij}v^j\di x^i \right)
=-\frac{\partial v^i}{\partial x^k}g_{ij}v^j\di x^k
-v^k \frac{\partial}{\partial x^k}\left(g_{ij}v^j\right)\di x^i.\]
Finally.
\begin{align*}
\delta_{\mathrm{diff}}\di\left(g_{ij}v^j\di x^i \right)
& =g^{mn}
\frac{\partial}{\partial x^n\partial x^m}\left(g_{kj}v^j\right)\di x^k
-g^{mn}\frac{\partial}{\partial x^n\partial x^k}\left(g_{mj}v^j\right)\di x^k
+\kappa
\end{align*}
where $\kappa$ is the 1-form,
\[\kappa = g^{mn}\frac{\partial}{\partial x^m}  \righthalfcup
\left(\frac{\partial}{\partial x^k}\left(g_{ij}v^j\right)\nabla_n(\di x^k \wedge\di x^i)\right)\]
and
\begin{align*}
\di \delta_{\mathrm{diff}} \left(g_{ij}v^j\di x^i \right)  
& = \left(
\frac{\partial g^{mn}}{\partial x^l} \frac{\partial}{\partial x^n}\left(g_{mj}v^j  \right)
+g^{mn} \frac{\partial}{\partial x^l\partial x^n}\left(g_{mj}v^j  \right)
-\frac{\partial}{\partial x^l}\left(g^{mn}g_{pj}v^j\Gamma^p_{nm} \right)   \right)
\di x^l
\end{align*}
Thus,
\begin{align*}
\Delta_{\mathrm{diff}} \left(g_{ij} v^j \di x^i\right)& = 
\Bigg(
\frac{\partial g^{mn}}{\partial x^k} \frac{\partial}{\partial x^n}\left(g_{mj}v^j  \right)
-\frac{\partial}{\partial x^k}\left(g^{mn}g_{pj}v^j\Gamma^p_{nm} \right)   
+
g^{mn}
\frac{\partial}{\partial x^n\partial x^m}\left(g_{kj}v^j\right)
\Bigg)\di x^k\\
&\qquad
+\kappa
\end{align*}
Thus, from \eqref{closed}, we deduce that $\xi$ is closed if,
\begin{align} 
\nonumber 0
&= \frac{\partial}{\partial x^l}\left(g_{ij}v^j\right) \di x^l\wedge \di x^i-\frac{1}{2}\kappa\wedge\di t\\
\nonumber &\qquad 
+\Bigg(
-\frac{\partial }{\partial t}\left(g_{kj}v^j\right)
-\frac12\frac{\partial g^{mn}}{\partial x^k} \frac{\partial}{\partial x^n}\left(g_{mj}v^j  \right)
-\frac12
g^{mn}
\frac{\partial}{\partial x^n\partial x^m}\left(g_{kj}v^j\right)
\\
&\qquad\qquad\qquad
+\frac{\partial v^i}{\partial x^k}g_{ij}v^j
+v^i \frac{\partial}{\partial x^i}\left(g_{kj}v^j\right)
-\frac12\frac{\partial}{\partial x^k}\left(g_{ij}v^jv^i\right)
\Bigg)\di x^k\wedge\di t\label{closed2}
\end{align}

The first term gives us,
\[0 = \frac{\partial}{\partial x^l}\left(g_{ij}v^j\right)
-\frac{\partial}{\partial x^i}\left(g_{lj}v^j\right).\]
Now we note that
\[\nabla_n(\di x^k\wedge\di x^i)
 = - \Gamma_{nl}^k \di x^l \wedge \di x^i + \Gamma_{nl}^i \di x^l \wedge \di x^k\]
 so that,
 \begin{align*}
\kappa
& = g^{mn}\frac{\partial}{\partial x^m} \righthalfcup
\left( -\frac{\partial}{\partial x^k}\left(g_{ij}v^j\right) \Gamma_{nl}^k \di x^l \wedge \di x^i +\frac{\partial}{\partial x^k}\left(g_{ij}v^j\right) \Gamma_{nl}^i \di x^l \wedge \di x^k\right)\\
& = g^{mn}\left(\left(\frac{\partial}{\partial x^i}\left(g_{kj} v^j\right)- \frac{\partial}{\partial x^k}\left(g_{ij} v^j\right)\right)\Gamma^{k}_{nm} 
+ \left(\frac{\partial}{\partial x^k}\left(g_{mj} v^j\right)- \frac{\partial}{\partial x^m}\left(g_{kj} v^j\right)\right)\Gamma^{k}_{ni} 
\right)\di x^i\\
&=0.
 \end{align*}

Thus we deduce from \eqref{closed2} that $\xi$ is closed if and only if,
\begin{equation}\label{iffn2}
\left\{\begin{array}{rl}
\displaystyle
\frac12
g^{mn}
\frac{\partial}{\partial x^n\partial x^m}\left(g_{kj}v^j\right) & = \displaystyle
-\frac12\frac{\partial g^{mn}}{\partial x^k} \frac{\partial}{\partial x^n}\left(g_{mj}v^j  \right)
+\frac12g_{ij}v^j\frac{\partial v^i}{\partial x^k}
\\[1em]
&\qquad
\displaystyle
+\frac12v^i \frac{\partial}{\partial x^i}\left(g_{kj}v^j\right)-\frac{\partial }{\partial t}\left(g_{kj}v^j\right)
\\[1em]
0& =\displaystyle \frac{\partial}{\partial x^l}\left(g_{ij}v^j\right)-\frac{\partial}{\partial x^i}\left(g_{lj}v^j\right)
 \end{array}\right.
\end{equation}
These can be compared with equations \eqref{iffna} and \eqref{differentiated}. 
The second equation in  \eqref{iffn2}  is the condition for the existence of a function $f$ such that,
\[\di f = g_{ij} v^j \di x^i\qquad \Leftrightarrow \qquad
\frac{\partial f}{\partial x^i} = g_{ij} v^j,\qquad \Leftrightarrow \qquad
\nabla_i f = v^i.\]
It is now clear where the gradient condition arises - it is an immediate consequence of  equation \eqref{closed}.

\begin{remark}
We also note that the PDEs \eqref{iffna} and \eqref{iffn2} are closely related to the heat equation
\[\frac{\partial u}{\partial t}  =\frac{1}{2} a_{ij} \frac{\partial^2 u}{\partial x^i\partial x^j},\qquad u(x,0) = u_0(x),\]
via the connection  $f = - \ln u$. They should also be compared with the results of Proposition \ref{prop6}.
\end{remark}

\section{Deformed covariant derivatives}\label{vcaisuvfiuyt}
First we need to say something about the fiberwise tensor product in differential geometry. The tensor product of two vector bundles over a manifold $M$ has value, at a point, of the tensor product of the corresponding vector spaces at that point. But the vector space of sections of the tensor product bundle is \textit{not} the simple vector space tensor product of the two vector spaces of sections. For example, if we were to take the tensor product of differential forms, the vector space $\Omega^nM\tens \Omega^mM$ is far larger than the sections of the corresponding tensor product vector bundle. To repair this (ignoring completions), we use the tensor product over the algebra of functions
$\Omega^nM\tens_{C^\infty(M)} \Omega^mM$, which is defined so that the following are equal, where $\xi,\eta$ are forms and $f$ is a function:
\begin{eqnarray}
\xi.f\tens\eta\,=\,\xi\tens f.\eta\ .
\end{eqnarray}
If we return to our DGA $(F_{\mathbb{R}}^*,\extd_\alpha,\wedge_\alpha)$, the fiberwise tensor product of forms is denoted by $F_{\mathbb{R}}^n\tens_\alpha F_{\mathbb{R}}^m$, and we have the relation
\begin{eqnarray}
(\xi\wedge_\alpha f)\tens\eta\,=\,\xi\tens (f\wedge_\alpha \eta)\ .
\end{eqnarray}
We use $F_{\mathbb{R}}^n\tens_0 F_{\mathbb{R}}^m$ to denote the $\alpha=0$ case. 
Now the properties of $\mathcal{I}$ mean that we have a well defined map
$(\mathcal{I}\tens \mathcal{I})(\xi\tens\eta)=\mathcal{I}(\xi)\tens \mathcal{I}(\eta)$,
\begin{eqnarray} \label{cvagkxyxz}
\mathcal{I}\tens \mathcal{I}: 
F_{\mathbb{R}}^n \tens_0 F_{\mathbb{R}}^m    \to F_{\mathbb{R}}^n \tens_\alpha F_{\mathbb{R}}^m \ .
\end{eqnarray}

To avoid considering vector fields and stay with just forms, we use the following definition of covariant derivative on the $n$-forms: A left covariant derivative is a map from $\Omega^nM$ to the fiberwise tensor product of $\Omega^1M$ and $\Omega^nM$ obeying the left Leibniz rule for a function $f$
\begin{eqnarray}
\nabla(f.\xi)\,=\,\extd f\tens\xi+f.\nabla(\xi)\ .
\end{eqnarray}
(For classical manifolds insisting on `left' is superfluous, but we may more generally have a noncommutative product.) 
To get the more usual directional derivative along a vector field version, just pair the vector field with the $\Omega^1M$ factor.

More generally, to deform a left covariant derivative $\nabla:F^n\to F^1\tens_{F^0} F^n$ we first add time to get a covariant derivative on $F_{\mathbb{R}}^n$,
\begin{eqnarray}
\nabla_0(\xi) &=& \nabla(\xi)+\extd t\tens \frac{\partial \xi}{\partial t}\ ,
\end{eqnarray}
where $\nabla(\xi)$ is defined pointwise in time, i.e.\ $\nabla(\xi)(t)=\nabla(\xi(t))$. Now we deform the $\tens_0$ used here to $\tens_\alpha$ by using (\ref{cvagkxyxz}), and define
\begin{eqnarray} \label{mjcfyjmycytu}
\nabla_\alpha\ =\ (\mathcal{I}\tens \mathcal{I})\,\nabla_0\,\mathcal{I}^{-1}: F_{\mathbb{R}}^n    \to 
F_{\mathbb{R}}^1 \tens_\alpha F_{\mathbb{R}}^n  \ .
\end{eqnarray}
Now we check the left Leibniz rule, remembering for functions that $ \mathcal{I}(f)=f$:
\begin{eqnarray*}
\nabla_\alpha(f\wedge_\alpha \mathcal{I}(\xi)) &=& \nabla_\alpha(\mathcal{I}(f.\xi)) \cr
&=& (\mathcal{I}\tens \mathcal{I})\,\nabla_0(f.\xi) \cr
&=&  (\mathcal{I}\tens \mathcal{I})(\extd_0(f)\tens \xi+f.\nabla_0(\xi)) \cr
&=& \mathcal{I}(\extd_0(f))\tens \mathcal{I}(\xi) + f\wedge_\alpha \nabla_\alpha(\mathcal{I}(\xi)) \cr
&=& \extd_\alpha(f)\tens \mathcal{I}(\xi) + f\wedge_\alpha \nabla_\alpha(\mathcal{I}(\xi))\ .
\end{eqnarray*}

\begin{example} Using the diffusion example in Section~\ref{cbhajksvckvkvv1}, 
(\ref{mjcfyjmycytu}) gives
\begin{eqnarray}
\nabla_\alpha(\extd x^k) &=& -\Gamma^k_{pq}\,\extd x^p\,{\tens}_\alpha\extd x^q \cr
&& +\, \alpha\,\left(\begin{array}{c} g^{i\,j}\big(\Gamma^{k}_{m\,s} \,\Gamma^{s}_{j\,i} +\dfrac{\partial \Gamma^{k}_{j\,i} }{\partial x^{m}}- \Gamma^{s}_{j\,m}\,\Gamma^{k}_{s\,i}-\Gamma^{s}_{i\,m}\,\Gamma^{k}_{j\,s}\big)\,\extd x^{m}\tens_{\alpha} \extd t\\
+\,\extd t \tens_{\alpha} g^{i\,j}\big( -\,\dfrac{\partial \Gamma^{k}_{i\,q} }{\partial x^{j}}+\Gamma^{k}_{p\,q}\,\Gamma^{p}_{j\,i}\big)\, \extd x^{q}\end{array}\right) \cr
&&+\, O(\alpha^2)\,\extd t\,{\tens}_{\alpha}\extd t\ .
\end{eqnarray}
The part given serves to illustrate that the combinations of Christoffel symbols do not appear in any particularly nice order. This is not surprising -- the modified product, including in the definition of the tensor product $\tens_\alpha$, makes everything different from the usual case. 
\end{example}

\section{The noncommutative sphere}  \label{vcadisuvuid}
From \cite{woron87}, for a parameter $ q\in \mathbb{R}  $, define the quantum group $ \mathbb{C}_{q}[SL_{2}] $ to have generators $ a,b,c,d $ with relations
\begin{eqnarray*}
&&b a=q a b\,,\,       ca = q a c\,, \,       d b=q b d\,,\,     d c=q c d\,,\,   c b=b c\,,\,    d a - a d = q(1-q^{-2})b c\,,\,\\
&&a d-q^{-1}b c=1\ .
\end{eqnarray*}
There is a star operation
\begin{equation*}
 a^{*} = d\ , \, d^{*} = a\ , \, c^{*}=-qb\ ,\, b^{*}=-q^{-1}\,c\ ,
\end{equation*}
which gives a deformed analogue of functions on the group $SU_2$. 
 By using the relations, any element of $ \mathbb{C}_{q}[SL_{2}] $ is a linear combination of
$ a^{n}\,b^{m}\, c^{p} $ or  $ d^{n}\,b^{m}\,c^{p} $, and we have
\begin{eqnarray*}
(a^{n}\,b^{m}\,c^{\,p})^{*} &=&(-1)^{p+m}\,q^{-n(p+m)+p-m}\,d^{\,n}\,b^{p}\,c^{m}\ ,\\
(d^{n}\,b^{\,p}\,c^{m})^{*} &=&(-1)^{p+m}\,q^{n(p+m)-p+m}\,a^{\,n}\,b^{m}\,c^{p}\ .
\end{eqnarray*} 
The standard quantum sphere   \cite{pod87} is given by a grading on $ \mathbb{C}_{q}[SL_{2}] $. Take $a,c$ to have grade $+1$ and $b,d$ to have grade $-1$. The functions on the quantum sphere $ S^{2}_{q} $ is the subalgebra of $ \mathbb{C}_{q}[SL_{2}] $ consisting of elements of grade zero. 

There is a differential calculus on $ S^{2}_{q} $ given by adding two elements $e^\pm$, where $e^\pm$ has grade $\pm 2$. $\Omega^1 S^{2}_{q} $ consists of $f_+.e^++f_-.e^-$ of total grade zero, where $f_\pm\in \mathbb{C}_{q}[SL_{2}] $. The differential $\extd$ is given by
\begin{equation*}
\extd a = q\,b\,e^{+},\ \extd b = a\,e^{-},\ \extd c = q\,d\,e^{+}, \  \extd d = c\,e^{-}    
\end{equation*}
The commutation relations of the 1-forms $e^\pm$ with the algebra are
\begin{equation*}
e^{\pm}\,a = q\,a\, e^{\pm},\ e^{\pm}\, b = q^{-1}\,b\,e^{\pm},\ e^{\pm}\,c = q\,c\,e^{\pm},\ e^{\pm}\,d = q^{-1}\,d\,e^{\pm}   
\end{equation*}
The wedge product of forms has the relations
\begin{equation*}
 q^{2}\,e^{+}\wedge e^{-}+e^{-}\wedge e^{+}=0,\quad e^{\pm}\wedge e^{\pm}=0.
\end{equation*}

It will be convenient to use the well known $ q $-integers, defined by
\begin{eqnarray}
[n]_{q}&=&\frac{1-q^{n}}{1-q}= 1+q+q^{2}+\dots+q^{n-1}\ .
\end{eqnarray} 
In what follows, we use $ [n]$ to mean $[n]_{q^{2}} $ for short, for example $ [0]= 0  $, $ [1]= 1  $ and $ [2]= 1+q^2 $. 
Also define the $q^2$-factorials $[n]!=[n][n-1][n-2]\dots[1]$ and the $q^2$-binomial coefficients by
\begin{eqnarray}
{p \brack r} \,=\, \frac{[p]!}{[r]!\ [p-r]!}\ .
\end{eqnarray}

\begin{prop}
$\extd$ applied to $S^2_q$ gives, for $n\ge 1$,
\begin{eqnarray*}\label{eq18}
\extd(a^{n}\,b^{m}\,c^{\,p}) &=&[p+n]\,q^{3-2p-n}\,e^{+}\,a^{\,n-1}\,b^{m+1}\,c^{\,p}
+[m]\,q^{-n-1}\,e^{-}\,a^{\,n+1}\,b^{m-1}\,c^{\,p}\nonumber\\
&&+\ [p]\,q^{4-2p-n}\,e^{+}\,a^{\,n-1}\,b^{m}\,c^{\,p-1}\ ,\cr
 \extd(d^{\,n}\,b^{m}\,c^{\,p}) &=&[m+n]\,q^{-n}\,e^{-}\,d^{\,n-1}\,b^{m}\,c^{\,p+1}
+ [p]\,q^{4-2p+n}\,e^{+}\,d^{\,n+1}\,b^{m}\,c^{\,p-1}\nonumber\\
 &&+\,[m]\,q^{n-1}\,e^{-}\,d^{\,n-1}\,b^{m-1}\,c^{\,p}\ ,\cr
 \extd(b^{m}\,c^{\,p}) &=&[m]\,q^{-1}\, e^{-}\,a\,b^{m-1}\,c^{\,p}+[p]\,q^{4-2p}\,e^{+}\,d\,b^{m}\,c^{\,p-1}.
\end{eqnarray*}
\end{prop}

\section{The eigenfunctions of the Laplace operator on $S^2_q$}
We repeat the calculations of Section~\ref{cbhajksvckvkvv1}  for the differential calculus on $S^2_q$ specified in Section~\ref{vcadisuvuid}. The covariant derivative of a $ 1$-form $ \xi =e^{+}\,f_{+}\,+e^{-}\,f_{-}$ is 
 \begin{eqnarray*}
  \nabla\,\xi& =& q^{\,-2}\,\extd f_{+}\tens e^{\,+}+q^{\,2}\,\extd f_{-}\tens e^{\,-}
  + q^{\,-2}\,f_{+}\, \nabla\,e^{\,+}+ q^{\,2}\,f_{-}\, \nabla\,e^{\,-}\ .
 \end{eqnarray*}
From \cite{BeMaStarRiem} we take the Levi Civita connection on $S^2_q$ given by the left covariant derivative specified by $\nabla e^\pm=0$. 
For the vector fields, it will be convenient to take $v^\pm$ (of grades $\mp2$) to be the dual basis to $e^\pm$ (i.e.\ evaluating $v^+$ on $e^+$ gives 1, and on $e^-$ gives 0). 
In term of vector fields we take the metric $g^{\mu\nu}$ to be, where $ \alpha $ and $\beta$ are real,
  \begin{eqnarray} \label{vchdksukjf}
 \alpha\,(v^{+} \tens v^{-}) +\beta\,(v^{-} \tens v^{+})\ .
  \end{eqnarray}
  The interior product of a vector field and a 1-form is simply taken to be evaluation, as given above.
Now we can use (\ref{cvdgsyjhtc}) to get the following result:

 \begin{prop} \label{vcdisuafvutyftuycudtuy}
 $ \delta $ applied to $ 1$-forms gives, for $n\ge 1$,
\begin{eqnarray*}
\delta (e^{+}\, a^{n}\, b^{m}\, c^{p}\,)&=&\alpha\,[m]\,q^{-n-3}\,a^{n+1}\,b^{m-1}\,c^{p}\\
\delta(e^{-}\, a^{n}\, b^{m}\,c^{p})&=&\beta\,[p]\,q^{6-2p-n}\, a^{n-1}b^{m}\,c^{p-1}+\beta\,[p+n]\,q^{5-2p-n}\,a^{n-1}\,b^{m+1}\,c^{p}\\
\delta(e^{+}\, d^{n}\, b^{m}\,c^{p})&=&\alpha\,[n+m]\,q^{-n-2}\,d^{\,n-1}\,b^{m}\,c^{\,p+1} + \alpha\,[m]\,q^{n-3}\,d^{\,n-1}\,b^{m-1}\,c^{\,p}\\
\delta (e^{-}\, d^{n}\, b^{m}\, c^{p}\,)&=&\beta\,[p]\,q^{6-2p+n}\,d^{\,n+1}\,b^{m}\,c^{\,p-1}\\
\delta(e^{+}\, b^{m}\,c^{p})&=&\alpha\,[m]\,q^{-3}\,a\,b^{m-1}\,c^{\,p}\\
\delta (e^{-}\, b^{m}\, c^{p}\,)&=&\beta\,[p]\,q^{6-2p}\,d\,b^{m}\,c^{\,p-1}\ .
 \end{eqnarray*}
 \end{prop}

\begin{prop}\label{prop2}
For functions in $S_{q}^{2} $,
\begin{eqnarray*}
\Delta(a^{n}\,b^{m}\,c^{p})&=&[m]\,q^{3-2m}(\beta+\alpha\,q^{-2})\big([m+1]\,a^{n}\,b^{m}\,c^{p}+\,q\,[p]\,a^{n}\,b^{m-1}\,c^{p-1}\big)\\
\Delta(d^{n}\,b^{m}\,c^{p})&=&[p]\,q^{3-2p}(\beta+\alpha\,q^{-2})\big([p+1]\,d^{\,n}\,b^{m}\,c^{\,p}+[m]\,q^{2n+1}\,d^{\,n}\,b^{m-1}\,c^{p-1}\big)\\
\Delta(b^{m}\,c^{p})&=&[p]\,q^{3-2p}(\beta+\alpha\,q^{-2})\big([p+1]\,b^{p}\,c^{\,p}+q\,[p]\,b^{p-1}\,c^{\,p-1}\big).
\end{eqnarray*}
\end{prop}

\begin{theorem}\label{Theor2}
For $S^{2}_{q}$ the eigenfunctions of $ \Delta $ are, where $x=b\,c$ and $p\ge 0$,  
\begin{eqnarray*}
&& a^{n}\,b^{n}\sum\limits_{r=0}^{p}
q^{(p-r)^{2}}   {p \brack r}   \, {2n+p+r \brack n+r}  
\,x^{r}\ ,\cr
&&d^{n}\,c^{n}\sum\limits_{r=0}^{p}
q^{(p-r)(2n+p-r)} {p \brack r}   \,  {2n+p+r  \brack n+r}           \,x^{r}     \ ,
\end{eqnarray*}
with eigenvalue $ (\beta+\alpha\,q^{-2}) [n+p+1]\,[n+p]\,q^{3-2n-2p}$ for $n\ge 1$, and
 \begin{equation*}\label{eq26}
\sum\limits_{r=0}^{p}
q^{(p-r)^{2}}  {p \brack r}    \,  {p+r \brack r}        \,x^r\ ,
\end{equation*}
with eigenvalue $ (\beta+\alpha\,q^{-2}) [p+1]\,[p]\,q^{3-2p}$.
\end{theorem}

\section{The interior product and higher forms}
The $ 2$-forms on $ S^{2}_{q} $ are $ (f\,e^{+}\wedge e^{-}) $ where $ f $ is an element of $ S^{2}_{q} $. To calculate $ \delta $ of $ 2$-forms  we need to evaluate vector field $ v^{+} $ and $ v^{-} $ on  $ 2$-forms. Classically this is the interior product of a vector field and an $ n $-form to give an $ n-1 $ form. in the absence of definite idea of how to do this in the noncommutative case, we define for $ \gamma,\epsilon\, \in \C $ 
\begin{eqnarray}  \label{bjfildavblk}
v^{+}\,\righthalfcup\,(e^{+}\wedge e^{-})\,=\,\gamma\,e^{-}\ ,\quad
v^{-}\,\righthalfcup\,(e^{+}\wedge e^{-})\,=\,-\,\epsilon\,e^{+}
\end{eqnarray}

\begin{prop}   \label{bncjdklsbkjh}
$\delta$ applied to a $ 2$-form on $ S^{2}_{q} $ gives 
 \begin{eqnarray*}
 \delta(a^{n}\,b^{m}\,c^{\,p}\,e^{+}\wedge e^{-})&=&\alpha\,e^{-}\,\gamma\,[m]\,q^{-n-5}\,a^{n+1}\,b^{m-1}\,c^{\,p}\\
 &&-\,\beta\,q^{7-2p-n}\,\epsilon\,e^{+}([p+n]\,a^{n-1}\,b^{m+1}\,c^{\,p}+[p]\,q\,a^{n-1}\,b^{m}\,c^{\,p-1})\\
 \delta(d^{n}\,b^{m}\,c^{\,p}\,e^{+}\wedge e^{-})&=&\alpha\,\gamma\, e^{-}\left(\begin{array}{c}[m+n]\,q^{-4-n}\,d^{\,n-1}\,b^{m}\,c^{\,p+1}\\+\,[m]\,q^{n-5}\,d^{\,n-1}\,b^{m-1}\,c^{\,p}\end{array}\right)\\
 &&-\,\beta\,q^{8-2p+n}\,\epsilon\, e^{+}\,[p]\,d^{\,n+1}\,b^{m}\,c^{\,p-1}\\
 \delta(b^{m}\,c^{\,p}\,e^{+}\wedge e^{-}) &=&\alpha\,q^{-5}\,\gamma\, e^{-}\,[m]\,a\,b^{m-1}\,c^{p}-\beta\,q^{8-2p}\,\epsilon\, e^{+}\,[p]\,d\,b^{m}\,c^{p-1}.\\
 \end{eqnarray*} 
\end{prop}

It is now possible to calculate the Laplace operator applied to 1-forms (in Proposition~\ref{cbvdsjk}) and to 2-forms (in Proposition~\ref{cbvdsjk11}). Note that the formula for
$\Delta$ on 1-forms is quite complicated. There are choices for 
$\gamma$ and $\epsilon$ in the interior product (\ref{bjfildavblk}) which will considerably simplify these formulae, but rather than merely justifying these values by simplifying the results of Proposition~\ref{cbvdsjk}, we shall see that they are predicted by Hodge theory.

\begin{prop}\label{cbvdsjk} On 1-forms $\Delta$ has the form
\begin{eqnarray*}
\Delta(e^{+}\,a^{n}\,b^{m}\,c^{p})&=&(\beta\,\epsilon+\alpha\,q^{-6})[m]\,q^{5-2p-2n}\,e^{+}([n+p+1]a^{n}\,b^{m}\,c^{\,p}+q\,[p]\,a^{n}\,b^{m-1}\,c^{\,p-1})\\
  &&+\alpha\,[m]\,[m-1]\,q^{-2n-5}\,e^{-}(1-\gamma\,q^{-2})a^{n+2}\,b^{m-2}\,c^{\,p}\\
\Delta(e^{-}\,a^{n}\,b^{m}\,c^{p})&=&\beta(1-\epsilon\,q^{4})a^{n-2}\,e^{+}
\left(\begin{array}{c}  [p+n]\,[p+n-1]\,q^{9-4p-2n}\,b^{m+2}\,c^{p}\\+\, [p]\,[p-1]\,q^{13-4p-2n}\,b^{m}\,c^{p-2}\\+\,  [p]\,[p+n-1](1+q^{2})q^{10-4p-2n}\,b^{m+1}\,c^{p-1}\end{array}\right)\\
&&+\,q^{5-2p-2n}(\beta+\alpha\,\gamma\,q^{-4})e^{-}\big([m+2]\,[m+1]\,a^{n}\,b^{m}\,c^{p}+[p]\,[m]\,q\,a^{n}\,b^{m-1}\,c^{p-1}\big)\\
\Delta(e^{-}\,a\,b^{m}\,c^{p})&=&\beta\,[p](1-\epsilon\,q^{4})\,e^{+}([p+1]\,q^{8-4p}\,d\,b^{m+1}\,c^{p-1}+[p-1]\,q^{11-4p}\,d\,b^{m}\,c^{p-2})\\
&&+\,q^{3-2p}(\beta+\alpha\,\gamma\,q^{-4})e^{-}([p+1]\,[m+1]\,a\,b^{m}\,c^{\,p}+[p]\,[m]\,q\,a\,b^{m-1}\,c^{\,p-1})\\
\Delta(e^{+}\,d^{n}\,b^{m}\,c^{p})&=&q^{5-2p}(\beta\,\epsilon+\alpha\,q^{-6}) e^{+}([n+m]\,[p+1]\,d^{n}\,b^{m}\,c^{p}+[m]\,[p]\,q^{2n+1}\,d^{n}\,b^{m-1}\,c^{p-1})\\
 &&+\,\alpha (1-\gamma\,q^{-2})d^{n-2}\,e^{-}\left(\begin{array}{c}[n+m]\,[n+m-1]\,q^{-1-2n}\,b^{m}\,c^{p+2}\\+\,[m]\,[m-1]\,q^{2n-5}\,b^{m-2}\,c^{p}\\
 +\,[m]\,q^{-2}(1+q^{-2})[m+n-1]\,b^{m-1}\,c^{p+1}\end{array}\right)\\
 \Delta(e^{+}\,d\,b^{m}\,c^{p})&=&q^{5-2p}(\beta\,\epsilon+\alpha\,q^{-6}) e^{+}([m+1]\,[p+1]\,d\,b^{m}\,c^{p}+[m]\,[p]\,q^{3}\,d\,b^{m-1}\,c^{p-1})\\
 &&+\,\alpha\,[m] (1-\gamma\,q^{-2})\,q^{-3}\,e^{-}\left(\begin{array}{c}[m+1]\,q^{-1}\,a\,b^{m-1}\,c^{p+1}+[m-1]\,a\,b^{m-2}\,c^{p}\end{array}\right)\\
\Delta(e^{-}\,d^{n}\,b^{m}\,c^{p})&=&\beta\,[p]\,[p-1]\,q^{13-4p+2n}\,e^{+}(1-\epsilon\,q^{4})d^{n+2}\,b^{m}\,c^{p-2}\\
 &&+\,[p]\,q^{5-2p}\,e^{-}(\beta+\alpha\,\gamma\,q^{-4})([m+n+1]\,d^{n}\,b^{m}\,c^{p}+[m]\,q^{2n+1}\,d^{n}\,b^{m-1}\,c^{p-1})\\
\Delta(e^{+}\,b^{m}\,c^{p})&=&[m]\,q^{5-2p}\,e^{+}(\beta\,\epsilon+\alpha\,q^{-6})([p+1]\,b^{m}\,c^{p}+[p]\,q\,b^{m-1}\,c^{p-1})\\
&&+\,\alpha\,[m]\,[m-1]\,q^{-5}\,e^{-}(1-\gamma\,q^{-2})a^{2}\,b^{m-2}\,c^{p}\\
\Delta(e^{-}\,b^{m}\,c^{p})&=&\beta\,[p]\,[p-1]\,q^{13-4p}\,e^{+}(1-\epsilon\,q^{4})d^{2}\,b^{m}\,c^{\,p-2}\\
&&+\,[p]\,q^{5-2p}(\beta+\alpha\,\gamma\,q^{-4})e^{-}([m+1]\,b^{m}\,c^{p}+[m]\,q\,b^{m-1}\,c^{p-1})
\end{eqnarray*}
\end{prop}

\begin{prop}  \label{cbvdsjk11}
For a function $ f\in S_{q}^{2} $,
\begin{eqnarray*}
\Delta(a^{n}\,b^{m}\,c^{p}\,e^{+}\wedge e^{-})&=&[m]\,q^{7-2m}\\
&&(\beta\,\epsilon+\alpha\,\gamma\,q^{-8})e^{+} \wedge e^{-}\,([m+1]\,a^{\,n}\,b^{m}\,c^{\,p}+[p]\,q\,a^{\,n}\,b^{m-1}\,c^{\,p-1})\\
\Delta(d^{n}\,b^{m}\,c^{p}\,e^{+}\wedge e^{-})&=&[p]\,\,q^{7-2p}\\
&&(\beta\,\epsilon+\alpha\,\gamma\,q^{-8})e^{+}\wedge e^{-}([p+1]\,\,d^{n}\,b^{m}\,c^{\,p}+[m]\,q^{2n+1}\,d^{n}\,b^{m-1}\,c^{\,p-1})\\
\Delta(b^{m}\,c^{p}\,e^{+}\wedge e^{-})&=&[p]\,q^{7-2p}\\
&&(\beta\,\epsilon+\alpha\,\gamma\,q^{-8})e^{+}\wedge e^{-}([p+1]\,b^{m}\,c^{p}+[p]\,q\,b^{m-1}\,c^{p-1})\ .
\end{eqnarray*}
\end{prop}

\section{The Hodge operation} \label{bcdilskjhcj}

Classically, the Hodge operator is a map $\diamondsuit : \Omega^{n}M \to \Omega^{\mathrm{top}-n}M $, where $ \mathrm{top} $ is the dimension of the manifold $ M $. As we already have a star operation on the algebra, it would be confusing to use star for the Hodge operation, as they can both be applied to the same objects. 
On the manifold $ M $ we have an inner product for $ \Omega^{n}M $ given for $ \eta, \xi \in \Omega^{n}M  $
\begin{equation}\label{eq22}
\langle \eta,\xi  \rangle=\int_{M}(\eta\wedge\diamondsuit\,\xi)\ .
\end{equation}
Here $ \eta\wedge \diamondsuit\,\xi $ is a $ n+(\mathrm{top}-n)=\mathrm{top} $ form.
By Stokes' theorem for orientated $M$, for $\xi\in \Omega^rM$ and $\eta\in \Omega^{r-1} M$, 
\begin{eqnarray}
0\,=\,\int_M\extd(\eta\wedge\diamondsuit\,\xi)\,=\,\int_M(\extd\eta)\wedge\diamondsuit\,\xi-(-1)^r\,
\int_M\eta\wedge\extd(\diamondsuit\,\xi)\ ,
\end{eqnarray}
and this can be rewritten (assuming the invertibility of $\diamondsuit$) as
\begin{eqnarray}
\<\extd\eta,\xi\> \,=\, (-1)^{r}\,\<\eta,\diamondsuit^{-1}\extd\diamondsuit\,\xi\>\ .
\end{eqnarray}
A method for ensuring that the Laplacian is positive, which is taken classically, is to take $\delta$ to the be the operator adjoint of $\extd$, i.e.\ for $\xi\in \Omega^rM$
\begin{eqnarray}   \label{vbcxyfukhgc}
\delta(\xi)\,=\,(-1)^{r}\,\diamondsuit^{-1}\extd\diamondsuit\,\xi\ .
\end{eqnarray}
We wish to use the Hodge operator on the noncommutative sphere. A reading of \cite{majspinsphere} will show that this operation has already been defined, but with different formulae to the ones we will use. 
Our purposes differ from those of \cite{majspinsphere}, where the Hodge operation was a module map. 
As we are interested in functional analysis and positivity, our Hodge operation is conjugate linear to make (\ref{eq22}) into a Hermitian inner product, in fact it is a module map from $\Omega^{n}S^2_q$ to the conjugate module of $\Omega^{\mathrm{top}-n}S^2_q$.

\begin{defn}\label{de1}
For the noncommutative sphere, we use $\diamondsuit\,:\Omega^{n}S^{2}_{q}\longrightarrow \Omega^{2-n}S^{2}_{q}$, and suppose that we can write 
\begin{eqnarray*}
\diamondsuit\,(e^{+}\,f)&=&Ke^{-}\,f^{*}\ ,\quad 
\diamondsuit\,(e^{-}\,f)\,=\,Me^{+}\,f^{*}\\
\diamondsuit\,(e^{+}\wedge e^{-} \,f)&=&-q^{-2}\,\diamondsuit\,(e^{-}\wedge e^{+} \,f)=-q^{-2}\,Lf^{*}\\
\diamondsuit\,(f)&=&Ne^{+}\wedge e^{-} \,f^{*}.
\end{eqnarray*}
\end{defn} 

Remember that a Hermitian inner product obeys $\<\xi,\eta\>^*=\<\eta,\xi\>$.

\begin{prop}\label{prop4}
The formula 
\begin{eqnarray}
\langle \eta,\xi  \rangle &=& \int_{M}(\eta\wedge\diamondsuit\,\xi).
\end{eqnarray}
gives a Hermitian inner product (not yet known to be positive) if $ K,M,L,N $ are imaginary.
\end{prop}

We are using the standard idea of integrating a top dimensional form, which corresponds to integrating a function by multiplying the fiunction by a non-vanishing top form. The resulting integral on $S^2_q$ can be calculated from the Haar integral on the quantum group $ \mathbb{C}_{q}[SL_{2}] $, but given our previous calculations, here it is easier to use the cohomological definition based on $H^2_{dR}(S^2_q)\cong\mathbb{C}$.
In either case, the result is, for $n\ge 1$,
\begin{eqnarray}
\int (b\,c)^{p}=\dfrac{(-q)^{p}}{[p+1]}\ , \quad \int a^{n}\,b^{m}\,c^{p}=0=\int d^{n}\,b^{m}\,c^{p}\ .
\end{eqnarray}
Now we consider the consequences of using the usual formula (see \cite{GHalgGeom}) for applying the Hodge operator twice. 

\begin{prop}\label{prop5}
The Hodge dual obeys 
\begin{eqnarray}\label{eq28}
 \diamondsuit\,\diamondsuit\,\eta &=& (-1)^{k(2-k)}\,\eta.
\end{eqnarray}
for all $ \eta \in \Omega^{k}S^{2}_{q} $ if and only if $ KM=1 $ and $ LN=q^{2} $ in the Definition (\ref{de1}).
\end{prop}

Now we compare the formula (\ref{vbcxyfukhgc}) for $\delta$ in terms of the Hodge operation with the formulae in Proposition~\ref{bncjdklsbkjh}, which refer to the interior product in (\ref{bjfildavblk}).

\begin{prop}\label{prop 1}
For $ \xi\in \Omega^{n}S^{2}_{q} $, the equation $(-1)^{n}\, \diamondsuit\,\delta\,\xi= \extd\,\diamondsuit\,\xi $ is satisfied, as long as the constants in (\ref{de1}) obey
\begin{eqnarray*}
\alpha\,=\,\frac{K}{N}\,q^{5}\,=\, \frac{L}{\gamma\,M}\,q^{5}  \ , \quad \beta\,=\,  -\,\frac{M}{N}\,q^{-3} 
\,=\,  -\,\dfrac{L}{\epsilon\,K}\, q^{-9} \ .
\end{eqnarray*}
\end{prop}

\begin{cor}\label{cor1}
To have $(-1)^{n}\, \diamondsuit\,\delta\,\xi= \extd\,\diamondsuit\,\xi $ and $ \diamondsuit\,\diamondsuit\,\xi = (-1)^{n(2-n)\,\xi} $, for $ \xi \in \Omega^{n}A $ for all $ n $, we need 
for $ K,L $ imaginary,
\begin{equation*}
\alpha=K\,L\,q^{3} , \quad \beta = -\,\frac{L}{K}\,q^{-5},\quad  \gamma=q^{2} , \quad \epsilon=q^{-4},
\end{equation*}
\end{cor}

\begin{remark}
The values of the constants derived from Hodge theory give a considerable simplification in the formulae for $\Delta$, as noted in Section~\ref{vchdjksjhfx}. While part of the reason may be that $\Delta$ is self adjoint, there may also be another reason, which might have a bearing on any probabilistic interpretation. To use the heat equation in Brownian motion, there is one basic fact -- the integral of a function is constant under the time evolution, corresponding to interpretation that no particle paths are lost or gained. In the inner product notation, the integral of a function $f$ is $\<f,1\>$, and saying that this is conserved under the heat equation evolution corresponds to showing that $\<\delta\extd f,1\>=0$. If we use the fact that $\delta$ is the adjoint of $\extd$, this is $\<\extd f,\extd(1)\>=0$. 
\end{remark}

\section{Harmonic Analysis for Forms on $ S^{2}_{q} $.}  \label{vchdjksjhfx}
The value of $ \gamma $ and $ \epsilon $ from Hodge theory in Corollary (\ref{cor1}) give a considerable simplification in the formula for $ \Delta $ on 1-forms, as can be seen by comparing Proposition~\ref{bvdijlsdbvil} to Proposition~\ref{cbvdsjk}.

\begin{prop}  \label{bvdijlsdbvil}
On 1-forms $\Delta$ has the form, with $\gamma=q^{2}$ and $\epsilon=q^{-4}$,
\begin{eqnarray*}
\Delta(e^{+}\,a^{n}\,b^{m}\,c^{p})&=&(\beta+\alpha\,q^{-2})[m]\,q^{1-2p-2n}\,e^{+}([n+p+1]a^{n}\,b^{m}\,c^{\,p}+q\,[p]\,a^{n}\,b^{m-1}\,c^{\,p-1})\\
\Delta(e^{-}\,a^{n}\,b^{m}\,c^{p})&=&q^{5-2p-2n}(\beta+\alpha\,q^{-2})e^{-}\big([m+2]\,[m+1]\,a^{n}\,b^{m}\,c^{p}+[p]\,[m]\,q\,a^{n}\,b^{m-1}\,c^{p-1}\big)\\
\Delta(e^{-}\,a\,b^{m}\,c^{p})&=&q^{3-2p}(\beta+\alpha\,\gamma\,q^{-4})e^{-}([p+1]\,[m+1]\,a\,b^{m}\,c^{\,p}+[p]\,[m]\,q\,a\,b^{m-1}\,c^{\,p-1})\\
\Delta(e^{+}\,d^{n}\,b^{m}\,c^{p})&=&q^{1-2p}(\beta+\alpha\,q^{-2}) e^{+}([n+m]\,[p+1]\,d^{n}\,b^{m}\,c^{p}+[m]\,[p]\,q^{2n+1}\,d^{n}\,b^{m-1}\,c^{p-1})\\
\Delta(e^{+}\,d\,b^{m}\,c^{p})&=&q^{1-2p}(\beta+\alpha\,q^{-2}) e^{+}([m+1]\,[p+1]\,d\,b^{m}\,c^{p}+[m]\,[p]\,q^{3}\,d\,b^{m-1}\,c^{p-1})\\
\Delta(e^{-}\,d^{n}\,b^{m}\,c^{p})&=&[p]\,q^{5-2p}\,e^{-}(\beta+\alpha\,q^{-2})([m+n+1]\,d^{n}\,b^{m}\,c^{p}+[m]\,q^{2n+1}\,d^{n}\,b^{m-1}\,c^{p-1})\\
\Delta(e^{+}\,b^{m}\,c^{p})&=&[m]\,q^{1-2p}\,e^{+}(\beta+\alpha\,q^{-2})([p+1]\,b^{m}\,c^{p}+[p]\,q\,b^{m-1}\,c^{p-1})\\
\Delta(e^{-}\,b^{m}\,c^{p})&=&[p]\,q^{5-2p}(\beta+\alpha\,q^{-2})e^{-}([m+1]\,b^{m}\,c^{p}+[m]\,q\,b^{m-1}\,c^{p-1})
\end{eqnarray*}
\end{prop}

Using Proposition~\ref{bvdijlsdbvil} we can find the eigen-forms and their eigenvalues, almost completing our study of real Hodge theory on the standard Podle\'s sphere:

\begin{theorem} \label{bcyudsaftuyiduy}
The list of eigenfunctions and eigenvalues of $ \Delta $ on 1-forms:
\begin{equation*}
e^{+}\,a^{n}\,b^{n+2}\sum\limits_{r=0}^{p}
q^{(p-r)^{2}}\,  {p \brack r} \, {2n+p+r+2 \brack n+r+2} \,x^{r}.
\end{equation*}
With eigenvalue $ (\beta+\alpha\,q^{-2}) [n+p+2]\,[n+p+1]\,q^{1-2n-2p}.
 $ 
 \begin{equation*}
e^{-}\,a^{n}\,b^{n-2}\sum\limits_{r=0}^{p}
q^{(p-r)^{2}}    \,  {p \brack r} \, {2n+p+r-2 \brack n+r-2}  \,   x^{r}\quad \text{where $ n\geq 2 $}.
\end{equation*}
With eigenvalue $ (\beta+\alpha\,q^{-2}) [n+p]\,[n+p-1]\,q^{5-2n-2p}.
 $
  \begin{equation*}
e^{+}\,d^{n}\,c^{n-2}\sum\limits_{r=0}^{m}
q^{(m-r)(2n+m-r)}   \,  {m \brack r} \, {2n+m+r-2 \brack n+r-2 }  \, 
x^{r}\quad \text{where $ n\geq 2 $}.
\end{equation*}
With eigenvalue $ (\beta+\alpha\,q^{-2}) [n+m]\,[n+m-1]\,q^{5-2n-2m}.
 $
 \begin{equation*}
e^{-}\,d^{n}\,c^{n+2}\sum\limits_{r=0}^{m}
q^{(m-r)(2n+m-r)}   \,  {m \brack r} \, {2n+m+r+2 \brack n+r+2}  \,  x^{r}.
\end{equation*}
With eigenvalue $ (\beta+\alpha\,q^{-2}) [m+n+2]\,[m+n+1]\,q^{5-2m-2n}.
 $
 \begin{equation*}
e^{+}\,b^{2}\sum\limits_{r=0}^{p}
q^{(p-r)^{2}}      \,  {p \brack r} \, {p+r+2 \brack r+2}  \, x^{r}.
\end{equation*}
With eigenvalue $ (\beta+\alpha\,q^{-2}) [p+2]\,[p+1]\,q^{1-2p}.
 $
\begin{equation*}
e^{-}\,c^{2}\sum\limits_{r=0}^{m}
q^{(m-r)^{2}}   \,  {m \brack r} \, {m+r+2 \brack r+2}  \,  x^{r}.
\end{equation*}
With eigenvalue $ (\beta+\alpha\,q^{-2}) [m+2]\,[m+1]\,q^{1-2m}.
 $
  \begin{equation*}
e^{+}\,d\,b\sum\limits_{r=0}^{p}
q^{(p-r)(2+p-r)}   \,  {p \brack r} \, {p+r+2 \brack r+1}  \,  x^{r}.
\end{equation*}
With eigenvalue $ (\beta+\alpha\,q^{-2}) [p+2]\,[p+1]\,q^{1-2p}.
 $
  \begin{equation*}
e^{-}\,a\,c\sum\limits_{r=0}^{m}
q^{(m-r)^{2}}   \,  {m \brack r} \, {m+r+2 \brack r+1}  \,  x^{r}.
\end{equation*}
With eigenvalue $ (\beta+\alpha\,q^{-2}) [m+2]\,[m+1]\,q^{1-2m}.
 $
\end{theorem}

\begin{theorem}  \label{bcydisugyu}
For all  $ f\in S_{q}^{2} $, we have 
\begin{equation*}
\Delta(f\,e^{+}\wedge e^{-})= (\Delta \,f)e^{+}\wedge e^{-}\ .
\end{equation*}
Thus the eigenvalues for $ \Delta $ on $ 2 $-forms are the same as those for $ 0 $-forms, and the corresponding eigenfunctions are $ e^{+}\wedge e^{-} $ times the $ 0 $-form eigenfunctions. 
\end{theorem}

\begin{remark}
Suppose that $(\beta+\alpha\,q^{-2})\neq 0$ and that $q> 0$ is not 1.
Checking Theorems~\ref{bcydisugyu}, \ref{bcyudsaftuyiduy} and \ref{Theor2} shows that the only harmonic forms (i.e.\ solutions of $\Delta\xi=0$) are constants for 0-forms, zero for 1-forms and constants times $e^+\wedge e^-$ for 2-forms. This fits the idea that we have one harmonic form for every de Rham cohomology class. 
The projection to the harmonic forms is taken to kill all the eigen-forms of $\Delta$ with non-zero eigenvalues. 

In terms of differential equations, from (\ref{bvhvbvzz}) we can set the $\extd t$ component of $\extd_\alpha(\xi)$ to be zero. This gives the diffusion or heat equation for forms,
 \begin{eqnarray}\label{bvhvcwbvzz}
\frac{\partial \xi}{\partial t}\,=\, -\,\alpha\,\Delta(\xi)\ ,
\end{eqnarray}
and we take the limit of $\xi$ as $t\to\infty$ and call it the projection applied to $\xi$. 
\end{remark}

To conclude, we should list the fundamental problems which have been solved in this paper by calculation in a particular example, but for which the general solutions in noncommutative geometry are not at all obvious. We should mention a non-problem first - the values of $\alpha$ and $\beta$ in the metric (\ref{vchdksukjf}) were never specified, simply because we never had to. 

\smallskip
1) The interior product: This was defined in (\ref{bjfildavblk}) with unknown values $\gamma,\epsilon$ which were  determined later. There is no general theory on how to do this, or on how to define this operation at all.

\smallskip
2) The Hodge operator: This was given in Definition~\ref{de1} with unknown values $K,L,M,N$ which were later determined using the assumption of some standard formulae for the Hodge operator, and these in turn gave the numerical values for $\gamma,\epsilon$. There is no general theory on how to do this to make (\ref{eq22}) a Hermitian inner product. 

\smallskip
3) The eigen-forms for $\Delta$ span a `dense' set: Classically this is implied by $(1+\Delta)^{-1}$ being a compact operator. We got round this by writing all the eigen-forms explicitly. In the work of Connes \cite{ConnesNCDG} a Dirac operator satisfying a similar compactness condition is a basic building block, but if we start from a DGA it is not so obvious what happens.

\end{document}